\documentclass[11pt,leqno]{amsart}

\usepackage{amsthm, amsmath, amssymb}
\usepackage{enumitem, geometry, caption, graphicx, array}
\usepackage{hyperref, url, fontenc, inputenc}
\usepackage{babel, booktabs, cite, float, footmisc, multicol, xcolor}
\usepackage{newtxmath, newtxtext, kvoptions, nag, ragged2e}
\usepackage{pdf14, pdftexcmds, xpatch, zref-base}

\usepackage{xy}
\xyoption{all}

\theoremstyle{plain}

\newcounter{theoremintro} 

\newtheorem{introtheorem}[theoremintro]{Theorem}

\newtheorem*{definition*}{Definition}
\newtheorem*{lemma*}{Lemma}
\newtheorem*{proposition*}{Proposition} 
\newtheorem*{theorem*}{Theorem} 
\newtheorem*{corollary*}{Corollary} 
\newtheorem*{conjecture*}{Conjecture}

\newtheorem{lemma}[subsection]{Lemma}
\newtheorem{proposition}[subsection]{Proposition}
\newtheorem{theorem}[subsection]{Theorem} 
\newtheorem{corollary}[subsection]{Corollary}
\newtheorem{remark}[subsection]{Remark}

\numberwithin{equation}{section}

\newcommand{\SO}{\mathrm{SO}}
\newcommand{\SU}{\mathrm{SU}}
\newcommand{\Sp}{\mathrm{Sp}}
\newcommand{\F}{\mathbb{F}}

\newcommand{\GL}{\mathrm{GL}}

\newcommand{\N}{\mathbb{N}}
 
\newcommand{\Ha}{\mathbb{H}}
\newcommand{\R}{\mathbb{R}}
\newcommand{\bC}{\mathbb{C}}

\newcommand{\op}{\mathrm{op}}

\newcommand{\s}[1]{\langle #1 \rangle}
\newcommand{\norm}[1]{\lVert #1 \rVert}

\newcommand{\hill}{\mathcal{H}}

\begin{document} 

\title[$\mathrm{Sp}(n ,1)$ admits a proper $1$-cocycle for a u.b. rep.]{$\mathrm{Sp}(n ,1)$ admits a proper 1-cocycle for a uniformly bounded representation}

\author[S. Nishikawa]{Shintaro Nishikawa}

\address{S.N.: Mathematisches Institut, Fachbereich Mathematik und Informatik der Universit\"at M\"unster, Einsteinstrasse 62, 48149 M\"unster, Germany.}
 
\email{snishika@uni-muenster.de}

\thanks{This work is partly supported by the Deutsche Forschungsgemeinschaft (DFG, German Research Foundation) under Germany's Excellence Strategy EXC 2044-390685587, Mathematics M\"unster: Dynamics-Geometry-Structure.}

\date{\today}

\subjclass[2010]{Primary 020, 022; Secondary 042}

\maketitle

\begin{abstract}
We verify Shalom's conjecture for the simple real-rank-one Lie group $\Sp(n ,1)$ for any $n$: i.e. we show that it admits a metrically proper affine action on a Hilbert space whose linear part is a uniformly bounded representation. We provide two different proofs. Both approaches crucially use results on uniformly bounded representations by Michael Cowling. The first approach is quite abstract: it uses an automatic-properness result of Shalom and requires almost no computations. The second approach is explicit: we deduce the properness of  cocycles from the non-continuity of a critical case of the Sobolev embedding. This work is inspired from Pierre Julg's work on the Baum--Connes conjecture for $\Sp(n,1)$.
\end{abstract}


\section{Introduction}
Let $G$ be a topological group. Recall that the first group cohomology group $H^1(G, \pi)=Z^1(G, \pi)/B^1(G, \pi)$ for a continuous linear representation $\pi$ on a topological vector space $V_\pi$ is defined as the vector space $Z^1(G, \pi)$ of $1$-cocycles on $V_\pi$ modulo the vector space $B^1(G, \pi)$ of co-boundaries. Here, a $1$-cocycle is a continuous function $b\colon G\mapsto V_\pi$ such that $b_{gh}=\pi(g)b_h+ b_g$ for $g, h$ in $G$. A $1$-cocycle of the form $b_g=v-\pi(g)v$ for $v\in V_\pi$ is called a co-boundary. 

If $\pi$ is an isometric representation of $G$ on a normed vector space $(V_\pi, \lVert \,\rVert_{V_\pi})$, a $1$-cocycle $b_g$ corresponds to an affine isometric action $v\mapsto \pi(g)v + b_g$ of $G$ on $V_\pi$. A $1$-cocycle $b_g$ is proper if $||b_g||_{V_\pi}\to \infty$ as $g\to \infty$ and this corresponds to a metrically proper affine isometric action. 

A second countable, locally compact group is called a-T-menable if it admits a proper $1$-cocycle for a unitary representation on a Hilbert space, that is if it admits a metrically proper affine action on a Hilbert space. As explained in \cite{CCJJV2001}, a-$T$-menability is equivalent to the Haagerup approximation property. A non-compact group with Kazhdan's property (T) cannot be a-T-menable because any affine isometric action of such a group has a fixed point.

Pierre Julg showed that the simple real-rank-one Lie groups $\SO_0(n, 1)$ ($n\geq2$) and $\SU(n,1)$ ($n\geq2$) are a-T-menable (see\cite[Section 1.4]{Julg1998}, \cite[Chapter 3]{CCJJV2001}). On the other hand, as he noted, the other simple real-rank-one Lie groups $\Sp(n,1)$ ($n\geq2$) and $F_{4(-20)}$ cannot admit  a proper $1$-cocycle because they  have Kazhdan's property (T) (\cite{Kostant}).

According to \cite[Section 3.9]{Nowak2015}, it is an unpublished result of Yehuda Shalom that the group $\Sp(n, 1)$ has a uniformly bounded representation $\pi$ on a Hilbert space in which there is a non-trivial $1$-cocycle: i.e. $H^1(G, \pi)\neq0$. Here, the term ``uniformly bounded'' means that there is $C>0$ such that the operator norm $||\pi(g)||$ of $\pi(g)$ is bounded by constant $C$ for all $g\in G$. However, it appears that this fact has not been proven explicitly in the literature. 
He also gave the following conjecture, which is now known as Shalom's conjecture:

\begin{conjecture*}(Yehuda Shalom, \cite[Conjecture 35]{Nowak2015}) Let $\Gamma$ be a non-elementary hyperbolic group. There exists a uniformly bounded representation $\pi$ of $\Gamma$ on a Hilbert space, for which $H^1(\Gamma, \pi)\neq 0$ and for which there exists a proper cocycle.
\end{conjecture*}

Our main result verifies his conjecture for $\Sp(n, 1)$ for any $n$. Let $G$ be any one of $\SO_0(n, 1)$ ($n\geq2$), $\SU(n,1)$ ($n\geq2$) and $\Sp(n,1)$. Let $P$ be a minimal parabolic subgroup of $G$. We denote by $\Omega^{\mathrm{top}}_{\int=0}(G/P)$, the subspace of the vector space $\Omega^{\mathrm{top}}(G/P)$ of top-degree forms on $G/P$ with zero integral.

\begin{introtheorem}\label{introthm-Shalom-conj} (Theorem \ref{thm-Shalom-conj}) Let $G$ be any one of simple real-rank-one Lie groups $\SO_0(n, 1)$ ($n\geq2$), $\SU(n,1)$ ($n\geq2$) and $\Sp(n,1)$ ($n\geq2$). Let $\pi$ be the natural representation of $G$ on $W_0=\Omega^{\mathrm{top}}_{\int=0}(G/P)$. Then, there is a Euclidean norm $\norm{\,}_{W_0}$ on $W_0$ for which $\pi$ is continuous and uniformly bounded. The $1$-cocycle $b$ in $\pi$ associated to the extension 
\begin{equation}\label{eq-ext-intro}
0 \to \Omega^{\mathrm{top}}_{\int=0}(G/P) \to \Omega^{\mathrm{top}}(G/P) \to \mathbb{C} \to 0
\end{equation}
is proper with respect to the norm $\norm{\,}_{W_0}$. In particular, if $\Gamma$ is any non-compact closed subgroup inside the Lie groups $\SO_0(n, 1)$ ($n\geq2$), $\SU(n,1)$ ($n\geq2$), $\Sp(n,1)$ ($n\geq2$) and their products, $\Gamma$ admits a proper $1$-cocycle in a uniformly bounded representation on a Hilbert space.
\end{introtheorem}

The proof we present in Section \ref{sec-abs-proof} is quite abstract, and it requires almost no computations. The following is a brief sketch of the proof: On one hand, we note that the extension \eqref{eq-ext-intro} defines a non-zero class $[b]$ in the first group cohomology $H^1(G, (\pi, W_0))$. Moreover, we note that the class $[b]$ remains non-zero in $H^1(G, (\pi, \overline{W_0}))$ for any Banach space completion $\overline{W_0}$ of $W_0$ for which $\pi$ is continuous (Lemma \ref{lem-nonzero}). On the other hand, we note that such a non-trivial $1$-cocycle $b$ is automatically proper when $\overline{W_0}$ is a reflexible Banach space (Proposition \ref{prop-criteria}). This is a simple consequence of, a generalization of, Shalom's result (\cite[Theorem 3.4]{Shalom2000}) which says that for a connected, simple real-rank-one Lie group $G$ with finite center, a $1$-cocycle $b$ in an isometric representation on a reflexible Banach space is proper if and only if it is not a co-boundary (Proposition \ref{prop-Shalom}).  This reduces our problem to finding a Euclidean norm $\norm{\,}_{W_0}$ on $W_0$ for which $\pi$ is continuous and uniformly bounded. We prove this by applying a deep result of Michael Cowling on uniformly bounded representations of simple real-rank-one Lie groups. The last half of Section \ref{sec-abs-proof} is devoted for explaining how we use his result to obtain such a norm $\norm{\,}_{W_0}$ on $W_0$ but the discussion is quite elementary.

In Section \ref{sec-exp-proof}, after some preliminaries in Section \ref{sec-prelim}, we provide another proof of Theorem \ref{introthm-Shalom-conj}, which does not use the automatic-properness result of Shalom. This means that we explicitly compute the norm of the $1$-cocycle $b$ with respect to the norm $\norm{\,}_{W_0}$ which we obtained as an application of Cowling's result. An interesting new feature is that we deduce the properness of the $1$-cocycle $b$ from the non-continuity of the Sobolev embedding at the critical degree.

In Section \ref{sec-Busemann}, we obtain another proper $1$-cocycle in a different uniformly bounded representation. We consider the Busemann cocycle (see \cite[Section 3.1]{CCJJV2001}):
 \[
\gamma_{x,y}(z) = \beta_z(x, y)= \lim_{z'\to z}\bigg{(} d_{Z}(z', y) - d_{Z}(z', x) \bigg{)} 
 \]
for $x, y$ in the symmetric space $Z=G/K$ and for $z\in G/P$. The function $\gamma\colon Z\times Z\to C^\infty(G/P)$ is $G$-equivariant and satisfies a cocycle relation: $c(x_0, x_1)+  c(x_1, x_2)=c(x_0, x_2)$. We may think this as a cocycle in the quotient space $C^\infty(G/P)/\bC1_{G/P}$ of $C^\infty(G/P)$ by the constant functions (this should best be viewed as the derivative of the Busemann cocycle).

\begin{introtheorem}(Theorem \ref{thm-Busemann})  Let $G$ be any one of $\SO_0(n, 1)$ $n\geq2$, $\SU(n,1)$ for $n\geq2$, and $\Sp(n,1)$ for $n\geq2$. Let $\pi_0$ be the natural representation of $G$ on $C^\infty(G/P)/\mathbb{C}1_{G/P}$. Then, there is a Euclidean norm $\lVert\,\,\rVert$ on $C^\infty(G/P)/\mathbb{C}1_{G/P}$ for which $\pi_0$ is continuous and uniformly bounded and for which the Busemann cocycle $\gamma$ is continuous and proper in sense that
\[
\norm{\gamma(x,y)} \to +\infty \,\,\, \text{as} \,\,\, d_Z(x,y)\to +\infty.
\]
\end{introtheorem}
We note that we can obtain a proper group $1$-cocycle $b$ in the uniformly bounded representation $\pi_0$ from $\gamma$ by defining $b_g=\gamma(gx_0, x_0) \in C^\infty(G/P)/\mathbb{C}1_{G/P}$. Here, unlike Theorem \ref{introthm-Shalom-conj}, the proof of this requires an explicit computation. The automatic-properness result is not applicable here because it is not a priori evident that this cocycle is not a co-boundary.

\section*{Acknowledgements} 
I would like to thank Michael Cowling, Cornelia Drutu,  Nigel Higson, Pierre Julg, Tim de Laat, Mikael de la Salle and Alain Valette for helpful comments.  
 
 \section{A proof of Shalom's conjecture for $\Sp(n, 1)$}\label{sec-abs-proof}

In this section, we provide a proof of Shalom's conjecture for $\Sp(n, 1)$. 

Let $G$ be a topological group. A continuous representation of $G$ on a topological vector space $V$ (over $\R$ or $\bC$) is a linear representation $\pi\colon G\to \GL(V)$ on $V$ such that the orbit map $G\times V \ni (g, v) \mapsto \pi(g)v \in V$ is continuous.  Note that for any continuous representation $\pi$ of $G$ on a normed vector space $(V, \norm{\,})$, $\pi(g)$ is always bounded\footnote{We avoid the term ``a bounded representation'' which sometimes refers to a Hilbert space representation whose matrix coefficients are bounded.}: for any $g\in G$, there is a constant bound $C=C(g)>0$ for the operator norm of $\pi(g)$: $\norm{\pi(g)}_\op \leq C$. It is called uniformly bounded if there is a constant $C>0$ such that $\norm{\pi(g)}_\op \leq C$ holds for all $g$ in $G$. Isometric representations are uniformly bounded with constant $C=1$. Any representation of a compact group on  a normed vector space $(V, \norm{\,})$ is uniformly bounded.

For any (continuous) representation $\pi$ of $G$ on $V$, a group $1$-cocycle for $G$ in $\pi$, or in $V$, is a continuous function $b\colon G\to V$ such that $b_{gh} = \pi(g)b_{h} + b_{g}$ for any $g, h$ in $G$. Such a $1$-cocycle $b$ is a co-boundary if $b_g= v - \pi(g)v$ for some vector $v\in V$. The first group cohomology group $H^1(G, \pi)$ is a vector space of $1$-cocycles of $G$ in $V$ modulo co-boundaries. Thus, $H^1(G, \pi)=0$ if and only if every $1$-cocycle in $V$ is a co-boundary.  Suppose $V$ is a normed vector space $(V, \norm{\,})$. A $1$-cocycle $b$ for $G$ in $V$ is proper if $\norm{b_g} \to \infty$ as $g\to \infty$. If $\pi$ is isometric (or uniformly bounded) and if $G$ is not compact, any proper $1$-cocycle defines a non-zero class in $H^1(G, \pi)$ since all co-boundaries are bounded.

We shall use the following generalization of Shalom's result:
\begin{proposition}\label{prop-Shalom}(\cite[Theorem 3.4]{Shalom2000}) Let $G$ be a connected, simple real-rank-one Lie group with finite center and $\pi$ be an isometric representation of $G$ on a reflexible Banach space. Suppose that a group $1$-cocycle $b$ for $G$ in $\pi$ is not a co-boundary. Then the cocycle $b$ is proper.
\end{proposition}
\begin{proof} The proof of Theorem 3.4 in \cite{Shalom2000} is still valid in this setting. Lemma 3.3 in \cite{Shalom2000} holds for isometric group actions on any reflexible Banach space by the Ryll--Nardzewski fixed point theorem \cite{Ryll1967} (see \cite[Section 3.4]{Nowak2015}).
\end{proof}

Note that any uniformly bounded representation $\pi$ of $G$ on a reflexible Banach space $(V, \norm{\,}_0)$ can be viewed as an isometric representation of $G$ on a reflexible Banach space $(V,  \norm{\,})$ where
\[
\norm{v}= \sup_{g\in G}\norm{\pi(g)v}_0
\]
for $v$ in $V$ (the two norms are equivalent). Therefore, thanks to Proposition \ref{prop-Shalom}, we have the following simple criteria to find a proper $1$-cocycle for a simple real-rank-one Lie group $G$:

\begin{lemma}\label{lem-criteria} Let $G$ be a connected, simple real-rank-one Lie group with finite center. Let $\pi$ be a uniformly bounded representation of $G$ on a normed vector space $V$ whose completion $\overline{V}$ is a reflexible Banach space. Suppose that a group $1$-cocycle $b$ for $G$ in $\pi$ is not a co-boundary in the completion $\overline{V}$. Then, the cocycle $b$ in $\pi$ is proper. \qed
\end{lemma}

Let $G$ be a real reductive Lie group as in \cite[Chapter 2]{WallachI}. We note that any connected, semi-simple Lie group with finite center is real reductive. Let $K$ be a maximal compact subgroup of $G$. Let $P$ be a minimal parabolic subgroup of $G$ and $W=\Omega^{\mathrm{top}}(G/P)$ be the vector space of top-degree forms on the compact manifold  $G/P$. We equip it with the natural representation $\pi$ of $G$, which is continuous with respect to the natural Fr\'echet topology. Let $W_0=\Omega^{\mathrm{top}}_{\int=0}(G/P)$ be a   $G$-invariant subspace of $W$ consisting of top-degree forms with zero integral.

We have a short exact sequence of representations of $G$:

\begin{equation}\label{eq-ext}
0 \to \Omega^{\mathrm{top}}_{\int=0}(G/P) \to \Omega^{\mathrm{top}}(G/P) \to \mathbb{C} \to 0,
\end{equation}
where $\bC$ is the trivial representation of $G$. Now take any norm $\norm{\,}_{W_0}$ on $W_0$ for which $\pi$ is continuous. Then, there is a norm $\norm{\,}_{W}$ on $W$, which is unique up to equivalences, so that we have the following short exact sequence of continuous representations of $G$ on Banach spaces:
\begin{equation}\label{eq-ext2}
0 \to \overline{(W_0, \norm{\,}_{W_0})} \to  \overline{(W, \norm{\,}_W)} \to \mathbb{C} \to 0,
\end{equation}
where $\overline{(\, , \,)}$ denotes the completion with respect to the given norm.  

\begin{lemma}\label{lem-noGsec} Suppose $G$ is not compact. For any norm $\norm{\,}_{W_0}$ on $W_0$ for which $\pi$ is continuous, the extension \eqref{eq-ext2} is non-trivial, i.e. there is no $G$-equivariant section to the quotient map $\overline{(W, \norm{\,}_W)} \to \mathbb{C}$.
\end{lemma} 
\begin{proof} The multiplicity of the trivial representation of $K$ in $W=\Omega^{\mathrm{top}}(G/P)$ is one. Here, the transitivity of the $K$-action on $G/P$ is used. The extension \eqref{eq-ext} splits $K$-equivariantly and $W_0=\Omega^{\mathrm{top}}_{\int=0}(G/P)$ does not contain any $K$-invariant vectors. We see that the projection $p_K\colon w\mapsto \int_K\pi(k)wd\mu(k)$ on $\overline{(W, \norm{\,}_W)}$ preserves $W_0$ and its zero on $W_0$, hence on its completion $\overline{(W_0, \norm{\,}_{W_0})}$. That is, the multiplicity of the trivial representation of $K$ in $ \overline{(W, \norm{\,}_W)}$ is still one so the $K$-invariant vectors in $ \overline{(W, \norm{\,}_W)}$ are multiples of the $K$-invariant volume form on $G/P$. On the other hand, there is no $G$-invariant volume form on $G/P$ unless $G$ is compact. Therefore, there is no nonzero $G$-invariant vector in $ \overline{(W, \norm{\,}_W)}$.
\end{proof} 

The corresponding to the extension \eqref{eq-ext} is a $1$-cocycle $b$ of $G$ in $W_0=\Omega^{\mathrm{top}}_{\int=0}(G/P)$. For any $v \in  W=\Omega^{\mathrm{top}}(G/P)$ which is a lift of $1\in \bC$, $b_g=v-\pi(g)v$ is a cocycle in $W_0$. This defines a unique class $[b]$ in $H^1(G, W_0)$ independently of the choice of a lift $v$. For example, we can take $v=\mathrm{vol}_{G/P}$ the $K$-invariant volume form on $G$ with respect to the canonical  $K$-invariant metric on $G/P$.

\begin{lemma}\label{lem-nonzero}  Suppose $G$ is not compact.  The $1$-cocycle $[b]\in H^1(G, W_0)$ is non-zero. For any norm $\norm{\,}_{W_0}$ on $W_0$ for which $\pi$ is continuous, the $1$-cocycle $[b]$ is nonzero in $H^1(G, \overline{(W_0, \norm{\,}_{W_0})})$.
\end{lemma}
\begin{proof} This follows because the extensions \eqref{eq-ext} and \eqref{eq-ext2} are non-trivial. 
\end{proof}

Combining this with Lemma \ref{lem-criteria}, we get the following:

\begin{proposition}\label{prop-criteria} Let $G$ be a connected, simple real-rank-one Lie group with finite center. Let $\pi$ be the natural representation of $G$ on $W_0=\Omega^{\mathrm{top}}_{\int=0}(G/P)$. Suppose that there is a norm $\norm{\,}_{W_0}$ on $W_0$ for which $\pi$ is continuous and uniformly bounded and for which the completion $\overline{(W_0, \norm{\,}_{W_0})}$ is a reflexible Banach space. Then, the $1$-cocycle $b$ in $\pi$ associated to the extension 
\[
0 \to \Omega^{\mathrm{top}}_{\int=0}(G/P) \to \Omega^{\mathrm{top}}(G/P) \to \mathbb{C} \to 0
\]
is proper with respect to the norm $\norm{\,}_{W_0}$.
\end{proposition}
\begin{proof}  By Lemma \ref{lem-nonzero}, $b$ is not a co-boundary in the reflexible Banach space $\overline{(W_0, \norm{\,}_{W_0})}$. By Lemma \ref{lem-criteria}, $b$ is proper.
\end{proof}

What we have shown is that if we find an appropriate norm on $W_0=\Omega^{\mathrm{top}}_{\int=0}(G/P)$ so that $\pi$ is continuous and uniformly bounded and so that the completion is a reflexible Banach space, then, we automatically obtain a proper $1$-cocycle for $G$ in $\overline{(W_0, \norm{\,}_{W_0})}$. In particular, if $W_0$ admits a Euclidean norm for which $\pi$ is continuous and uniformly bounded, then Shalom's conjecture holds for $G$ (if $G$ is a connected, simple real-rank-one Lie group with finite center). For connected, simple real-rank-one Lie groups $G=\SO_0(n, 1)$ ($n\geq2$), $\SU(n,1)$ ($n\geq2$), $\Sp(n,1)$ ($n\geq2$), this is indeed the case:

\begin{theorem}\label{thm-W0-ub} Let $G$ be any one of simple real-rank-one Lie groups $\SO_0(n, 1)$ ($n\geq2$), $\SU(n,1)$ ($n\geq2$) and $\Sp(n,1)$ ($n\geq2$). Let $\pi$ be the natural representation of $G$ on $W_0=\Omega^{\mathrm{top}}_{\int=0}(G/P)$. Then, there is a Euclidean norm on $W_0$ for which $\pi$ is continuous and uniformly bounded.
\end{theorem} 

\begin{theorem}\label{thm-Shalom-conj} Let $G$ be any one of simple real-rank-one Lie groups $\SO_0(n, 1)$ ($n\geq2$), $\SU(n,1)$ ($n\geq2$) and $\Sp(n,1)$ ($n\geq2$). Let $\pi$ be the natural representation of $G$ on $W_0=\Omega^{\mathrm{top}}_{\int=0}(G/P)$. Then, there is a Euclidean norm $\norm{\,}_{W_0}$ on $W_0$ for which $\pi$ is continuous and uniformly bounded. The $1$-cocycle $b$ in $\pi$ associated to the extension 
\[
0 \to \Omega^{\mathrm{top}}_{\int=0}(G/P) \to \Omega^{\mathrm{top}}(G/P) \to \mathbb{C} \to 0
\]
is proper with respect to the norm $\norm{\,}_{W_0}$. In particular, if $\Gamma$ is any non-compact closed subgroup inside the Lie groups $\SO_0(n, 1)$ ($n\geq2$), $\SU(n,1)$ ($n\geq2$), $\Sp(n,1)$ ($n\geq2$) and their products, $\Gamma$ admits a proper $1$-cocycle in a uniformly bounded representation on a Hilbert space.
\end{theorem}
\begin{proof} Combine Theorem \ref{thm-W0-ub} with Proposition \ref{prop-criteria}. 
\end{proof}

The rest of this section is devoted to the proof of Theorem \ref{thm-W0-ub}. The proof is a simple application of a theorem of Michael Cowling on uniformly bounded representations of simple rank-one Lie groups. We start by recalling standard terminologies. Let $G$ be any one of $\SO_0(n, 1)$ ($n\geq2$), $\SU(n,1)$ ($n\geq2$) and $\Sp(n,1)$ ($n\geq2$), $K$ be a maximal compact subgroup of $G$ and $P$ be a minimal parabolic subgroups with the Langlands decomposition $P=MAN$. The corresponding Lie algebras are denoted by $\mathfrak{g}$, $\mathfrak{a}$,  $\mathfrak{n}$, $\ldots$ as usual. 

Let $\rho$ be the half-sum of the roots of $\mathfrak{a}$ on $\mathfrak{n}$ (note they are the positive roots). For any unitary irreducible (finite-dimensional) representation $\mu$ of $M$ on $\hill_\mu$ and for any complex number $\lambda\in \bC$, we consider the vector space $\mathcal{I}_{\mu, \lambda}$ of $\hill_\mu$-valued measurable functions $f$ on $G$ satisfying
\[
f(gman)=\mu^{-1}(m)\mathrm{exp}( -(1 + \lambda)\rho)\log(a)) f(g)
\]
for any $man$ in $MAN=P$ where $\log$ is the inverse of the exponential map $\mathrm{exp}\colon \mathfrak{a} \to A$. The group $G$ acts on $\mathcal{I}_{\mu, \lambda}$ and on its subspace $\mathcal{I}^\infty_{\mu, \lambda}$ consisting of smooth functions by the left-translation. We denote this representation on $\mathcal{I}^\infty_{\mu, \lambda}$ by $\mathrm{ind}_{\mu, \lambda}$. This is the principal series representation associated to $\mu$ and $\lambda$. The subspace $\mathcal{I}^\infty_{\mu, \lambda}$ is naturally identified as the space of the smooth sections of the vector bundle $G\times_P\hill_\mu$ over $G/P$ where $\hill_\mu$ is regarded as the representation of $P$ by letting $a\in A$ act by the scalar $\mathrm{exp}((1+\lambda)\rho(\log a))$ and by letting $N$ act trivially. Note when $\mu$ is the trivial representation $1_M$ and when $\lambda=\pm1$, $\mathcal{I}^\infty_{\mu, \lambda}=\mathcal{I}^\infty_{1_M, \pm1}$ are naturally identified as $C^\infty(G/P)$ for $\lambda=-1$, the smooth functions on $G/P$ with the left-translation $G$-action and $\Omega^\mathrm{top}(G/P)$ for $\lambda=1$, the top-degree forms on $G/P$ with the natural $G$-action.   We recall that we have $G=KP$ and $G/P=K/M$ canonically.

For any $1\leq p \leq +\infty$, there is a canonical $p$-norm on $\mathcal{I}_{\mu, \lambda}$: 
\[
 \lVert f\rVert_p^{(K)} = \bigg{(}\int_K \lVert f(k)\rVert^pd\mu_K(k)\bigg{)}^{1/p} = \bigg{(}\int_{K/M} \lVert f(kM)\rVert^pd\mu_{K/M}(kM)\bigg{)}^{1/p}.
\]
If $p=\infty$, the norm is defined as the essential supremum $\mathrm{ess}\sup_{k\in K} ||f(k)||$. Here, $\mu_K$ is the normalized Haar measure on $K$.

\begin{proposition}\label{prop_lp}(\cite[Lemma 5.2]{Cowling2010}) Suppose $\lambda$ is in the tube $T$ given by
\[
T = \{ \lambda \in \bC \mid \mathrm{Re}(\lambda) \in [-1, 1] \}
\]
and let $p\in [1, +\infty]$ be given by the formula
\[
1/p= \mathrm{Re}(\lambda)/2 + 1/2.
\]
Then, $G$ acts isometrically on $\mathcal{I}_{\mu, \lambda}^{L^p}$, the space of functions $f$ in $\mathcal{I}_{\mu, \lambda}$ for which $\lVert f\rVert_p^{(K)}$ is finite, equipped with the same norm.  \end{proposition}
The following defines a non-degenerate, $G$-equivariant paring between $\mathcal{I}^\infty_{\mu^\ast, \lambda}$ and $\mathcal{I}^\infty_{\mu, -\lambda}$:
\[
\mathcal{I}^\infty_{\mu^\ast, -\lambda} \times \mathcal{I}^\infty_{\mu, \lambda} \ni (f_1, f_2) \mapsto \int_K f_1(k)\cdot f_2(k) \mu_K(k).
\] 
Thus, the representations $\mathrm{ind}_{\mu, \lambda}$ and $\mathrm{ind}_{\mu^\ast, -\lambda}$ are naturally dual to each other. 
 
    \begin{theorem}\label{thm-Cowling-ub}(\cite[Theorem 7.1]{Cowling2010}, c.f. \cite{ACD2004}, see also \cite[Theorem 26, Corollary 27]{Julg2019}) Suppose $\lambda$ is inside the tube $T$, namely suppose $\mathrm{Re}(\lambda) \in (-1, 1)$. Then, there is a Euclidean norm on $\mathcal{I}^\infty_{\mu, \lambda}$ for which $\mathrm{ind}_{\mu, \lambda}$ is continuous and uniformly bounded.
   \end{theorem}
   
  We discuss some corollary to this theorem. This is of independent interest and there would be no problem for the reader to skip this part.  Recall that an admissible $(\mathfrak{g}, K)$-module (see \cite[3.3]{WallachI}) is a representation of the universal enveloping algebra of $\mathfrak{g}$ on a complex vector space such that
  \begin{enumerate}
  \item the representation of $\mathfrak{k}$ exponentiates to $K$;
  \item every vector lies in a finite-dimensional $K$-stable subspace (such a vector is called $K$-finite);
  \item every irreducible representation of $K$ occurs with finite multiplicity.
  \end{enumerate}
  
  A continuous representation of $G$ on a Hilbert space is called admissible if every irreducible representation of $K$ occurs with finite multiplicity. If $\pi$ is a continuous admissible representation on a Hilbert space $\hill$, the subspace $\hill_K$ of K-finite vectors is an admissible $(\mathfrak{g}, K)$-module. Every admissible, irreducible (more generally, finitely generated)  $(\mathfrak{g}, K)$-module $V$ is equivalent to $\hill_K$ for some continuous representation  $\pi$ on a Hilbert space $\hill$ \cite[3.8.3]{WallachI}. Two continuous Hilbert space representations are (infinitesimally) equivalent if the underlying $(\mathfrak{g}, K)$-modules are equivalent.

  \begin{corollary}\label{cor-TFAE} Let $G$ be any one of $\SO_0(n, 1)$ ($n\geq2$), $\SU(n,1)$ ($n\geq2$) and $\Sp(n,1)$ ($n\geq2$). Let $\pi$ be a non-trivial irreducible admissible representation of $G$ on a Hilbert space $\hill$. Then, the following are equivalent:
  \begin{enumerate}
  \item $\pi$ is infinitesimally equivalent to a uniformly bounded representation of $G$ on a Hilbert space;
  \item $\pi$ is infinitesimally equivalent either to an irreducible representation which is weakly contained in the left-regular representation or to the unique irreducible quotient $J_{\sigma, \lambda}$ of $\mathrm{ind}_{\sigma, \lambda}$ for some irreducible representation $\sigma$ of $M$ and for $\lambda$ with $0<\mathrm{Re}(\lambda)<1$;
  \item For any $u, v$ in $\hill_K$, the matrix coefficient $c_{u, v}\colon g\mapsto \s{gu, v}$ is $C_0$-function on $G$;
  \item For any $u, v$ in $\hill_K$, the matrix coefficient  $c_{u, v}\colon g\mapsto \s{gu, v}$ is $L^p$-function on $G$ for some $0<p<\infty$.
    \end{enumerate}
   \end{corollary}
\begin{proof}  The implication (1) $\implies$ (2) is  \cite[IV. Theorem 5.2]{BW} specialized to the rank-one case. The implication (2) $\implies$ (3) and (4) is proven in \cite[IV. Theorem 5.4]{BW} for any connected, simple Lie group with finite center. (2) $\implies$ (1) by Theorem \ref{thm-Cowling-ub}. 
 On the other hand, any irreducible admissible irreducible $(\mathfrak{g}, K)$-module is equivalent to the Langlands quotient $J_{P, \sigma, \nu}$ for a uniquely determined Langlands data $(P, \sigma, \nu)$ as in \cite[IV. Theorem 4.11]{BW}. If the condition (2) is not satisfied, $\pi$ is infinitesimally equivalent to the unique irreducible quotient $J_{\sigma, \lambda}$ of $\mathrm{ind}_{\sigma, \lambda}$ for $\mathrm{Re}(\lambda)\geq1$. As in the proof of \cite[IV. Theorem 5.2]{BW}, this implies that the matrix coefficients $c_{v_1, v_2}$ for $\pi$ satisfy
 \[
 \lim_{t\to \infty}e^{-t\alpha}c_{v_1, v_2}(a_t) = L(v_1, v_2)
 \]
 for any $v_1, v_2$ in $\hill_K$ for some nonzero sesquilinear form $L$ on $\hill_K$ and for some $\alpha\geq 0$ where $a_t\in A^+=\exp(\mathfrak{a}^+)$ for $t\geq0$ (the parametrization in $t$ is with respect to a fixed nonzero positive functional on $\mathfrak{a})$. It is not hard to see that such $c_{v_1, v_2}$ can neither be a $C_0$-function nor a $L^p$-function on $G$. This shows  (3) or (4) $\implies$ (2).
 \end{proof}
   
 We now give a proof of Theorem \ref{thm-W0-ub}. Let us first consider the principal series representation $\mathrm{ind}_{1_M, -1}$ on $\mathcal{I}_{1_M, -1}=C^\infty(G/P)$. The tangent bundle $T(G/P)$ of $G/P$ contains a $G$-equivariant subbundle $E$ of codimension $1$ for $G=\SU(n, 1)$ and of codimension $3$ for $G=\Sp(n, 1)$. The fiber of the cotangent bundle $T^\ast (G/P)$ at $P$ is naturally identified as the Lie algebra $\mathfrak{n}\cong (\mathfrak{g}/\mathfrak{p})^\ast$ where the isomorphism is via the Killing form. The nilpotent Lie algebra $\mathfrak{n}$ contains the center $\mathfrak{z}$ of dimension $1$ for $G=\SU(n, 1)$ and of dimension $3$ for $G=\Sp(n, 1)$. This defines a $G$-equivariant subbundle $F$ of $T^\ast (G/P)$ and the subbundle $E$ of $T(G/P)$ is defined as the annihilator $F^\perp$ of $F$  (see \cite{Julg2019} for some detail). When $G=\SO_0(n, 1)$, we set $E=T(G/P)$.

We let $\Gamma(E^\ast)$ be the complexified vector space of smooth sections of the bundle $E^\ast$ and 
\[
d_E=\,\,\, \mid_E \circ d \colon C^\infty(G/P)\to \Omega^1(G/P) \to \Gamma(E^\ast)
\]
be the composition of the de-Rham differential $d$ and the restriction of one-forms defined on $T(G/P)$ to the subbundle $E$. With respect to the natural $G$-action on $C^\infty(G/P)$ and $\Gamma(E^\ast)$, $d_E$ is $G$-equivariant. The kernel of $d_E$ is spanned by the constant function $1_{G/P}$ since the vector fields along $E$ generate all the vector fields on $G/P$.

\begin{lemma}\label{lem-principal}  Let $G$ be one of $\SO_0(n ,1)$ for $n\geq3$, $\SU(n ,1)$ for $n\geq2$ and $\Sp(n, 1)$ for $n\geq2$. Then, $\Gamma(E^\ast)$ equipped with the natural representation of $G$ is naturally identified as the principal series representation $\mathrm{ind}_{\mu, \lambda}$ for some finite-dimensional unitary representation $\mu$ of $M$ and $\lambda=-1+2/r$ where $r=n-1$ for $\SO_0(n ,1)$, $r=2n$ for $\SU(n ,1)$ and $r=4n+2$ for $\Sp(n, 1)$.  
\end{lemma}
\begin{proof} For any $G$, $\Gamma(E^\ast)$ is canonically identified as a representation induced from the representation $\mu$ of $P$ on the fiber of $E^*$ at $P$, which is the adjoint representation of $P=MAN$ on the complexification of $(\mathfrak{n}/\mathfrak{z})$ where $\mathfrak{z}=0$ for $\SO_0(n, 1)$. In each case, $A$ acts by the character $a\mapsto \mathrm{exp}(\frac2r\rho\log a)$ and $N$ acts by the identity. For $\SO_0(n, 1)$, the representation $\mu$ of $M=\SO(n-1)$ is the standard representation on $\bC^{n-1}$. For $G=\SU(n, 1)$, the complexification of $(\mathfrak{n}/\mathfrak{z}) \cong \bC^{n-1}$ decomposes into two irreducible unitary representations of $M$. For $G=\Sp(n , 1)$, the representation $\mu$ of $M$ is an irreducible unitary representation on $\Ha^{n-1}\otimes_\R\bC$.
\end{proof}

\begin{corollary}\label{cor-East-ub} Let $G$ be one of $\SO_0(n ,1)$ for $n\geq3$, $\SU(n ,1)$ for $n\geq2$ and $\Sp(n, 1)$ for $n\geq2$. Then, there is a Euclidean norm on $\Gamma(E^\ast)$ for which  the natural representation of $G$ is continuous and uniformly bounded.
\end{corollary}
\begin{proof} Combine Lemma \ref{lem-principal} and Theorem \ref{thm-Cowling-ub}. 
\end{proof} 

\begin{proposition}\label{prop-pi0}  Let $G$ be one of $\SO_0(n ,1)$ for $n\geq2$, $\SU(n ,1)$ for $n\geq2$ and $\Sp(n, 1)$ for $n\geq2$. Let $\pi$ be the principal series representation $\mathrm{ind}_{1_M, -1}$ on $\mathcal{I}^\infty_{1_M, -1}$: the natural representation of $G$ on $C^\infty(G/P)$. Consider the trivial sub-representation $\bC1_{G/P}$ of $\pi$, consisting of constant functions. Let $\pi_0$ be the quotient representation of $\pi$ on $C^\infty(G/P)/\bC1_{G/P}$. Then, there is a Euclidean norm on $C^\infty(G/P)/\bC1_{G/P}$ for which $\pi_0$ is continuous and uniformly bounded.
\end{proposition}
\begin{proof} The de-Rham differential $d_E$ induces a $G$-equivariant embedding of the representation $\pi_0$ on $C^\infty(G/P)/\bC1_{G/P}$ to the natural representation on  $\Gamma(E^\ast)$. By Corollary \ref{cor-East-ub}, the latter has a desired Euclidean norm except when $G=\SO_0(2 ,1)$. In the exceptional case $G=\SO_0(2 ,1)$, the representation $\pi_0$ on $C^\infty(G/P)/\bC1_{G/Px}$ is canonically equivalent, through the Poisson transform, to the square-integrable representation on $L^2$-harmonic $1$-forms on $G/K$ which is a unitary representation.
\end{proof}

\begin{proof}[~Proof of Theorem \ref{thm-W0-ub}] By Proposition \ref{prop-pi0}, the natural representation $\pi_0$ on $C^\infty(G/P)/\bC1_{G/P}$ admits a Euclidean norm for which $\pi_0$ is continuous and uniformly bounded. On the other hand, the representation $\pi$ on $W_0=\Omega^{\mathrm{top}}_{\int=0}(G/P)$ is a dual representation of $\pi_0$ through the non-degenerate pairing induced from the paring: 
\[
\Omega^{\mathrm{top}}(G/P) \times C^\infty(G/P) \ni (\omega, f) \mapsto \int_{G/P} f\omega.
\] 
The assertion is now immediate.
\end{proof}

 \section{Preliminaries for an alternative proof}\label{sec-prelim}
We will provide an alternative proof of Shalom's conjecture for $\Sp(n, 1)$. In this proof, we will prove explicitly the properness of the cocycle $b$ as in Theorem \ref{thm-Shalom-conj} meaning that we will not use the automatic properness result of Proposition \ref{prop-criteria}. This requires us to understand the content of Cowling's theorem (Theorem \ref{thm-Cowling-ub}) in a much deeper level. Because of this, we refresh our argument from scratch, starting with some preliminaries.

References for this preliminary section are \cite{Folland1975}, \cite{Cowling2010}, \cite{ACD2004} and \cite{Julg2019}.
   \subsection*{Lie groups $O(q)$}
   We shall follow the notations used in \cite{Cowling2010} mostly but not entirely: for example, we define the Lie group $O(q)$ as matrices of right-linear transformations on the right-vector space $\F^{n+1}$ although the left-linear convention was used in \cite{Cowling2010}. 
   
   Let $\F=\R, \bC, \Ha$ be the field of real numbers, complex numbers or quaternions with the natural inclusions between them. We write an element $z$ in $\F$ as
   \[
   z=s+t{\bold i} + u{\bold j} + v{\bold k}
   \]
   and write
   \[
   \bar z= s-t{\bold i} - u{\bold j} - v{\bold k},\,\,\,  |z|=(\bar z z)^{1/2},\,\,\, \mathrm{Re}(z)=\frac{z+\bar z}{2}, \,\,\, \mathrm{Im}(z)=\frac{z-\bar z}{2}.
   \]
   The imaginary part $\mathrm{Im}(\F)$ of $\F$ consists of the range of $\mathrm{Im}$ which is a vector subspace of $\F$ over $\R$.  We consider the right vector space $\F^{n+1}$ over $\F$ with the standard basis $e_0, e_1, \cdots, e_n$. The coordinates of an element $z$ in $\F^{n+1}$ with respect to the basis $(e_j)_{j=0}^{j=n}$ is written as $z=(z_j)_{j=0}^{j=n}$ for $z_j$ in $\F$. A sesquilinear form $q$ on $\F^{n+1}$ is given by
   \[
   q(z, w)=-\bar z_0 w_0 + \sum_{j=1}^{j=n}\bar z_j w_j
   \]
   for $z, w$ in $\F^{n+1}$. We consider the group $O(q)$ of $(n+1)\times (n+1)$ matrices over $\F$ which acts on $\F^{n+1}$ from left and preserves the quadratic form $q$, i.e. $A$ in $O(q)$ satisfies 
   \[
   q(Az, Aw)=q(z, w) \,\,\, \text{for any $z=\s{z_0, \cdots, z_n}^T$, $w=\s{w_0, \cdots, w_n}^T$ in $\F^{n+1}$}.
   \] 
   The matrix group $O(q)$ is a matrix Lie group and is connected unless $\F=\R$. The Lie group $\SO_0(n, 1)$ is the connected component of the identity of $O(q)$ for $\F=\R$, $\SU(n, 1)$ is $O(q)\cap SL(n+1, \bC)$ for $\F=\bC$ and $\Sp(n, 1)$ is $O(q)$ for $\F=\Ha$. From now on, a group $G$ is one of $\SO_0(n, 1)$ ($n\geq2$), $\SU(n, 1)$ ($n\geq2$) and $\Sp(n, 1)$ ($n\geq2$) \footnote{We used the definition for $\SO_0(1, n)$, $\SU(1, n)$, $\Sp(1, n)$ but the difference is merely in whether we take $e_0$ as the first coordinate or as the last coordinate}. The Lie algebra $\mathfrak{g}$ of $G$ consists of matrices of the form
   \[
   \begin{bmatrix} X & x^\ast \\ x & Y \end{bmatrix}
   \]
   where $X$ is in $\mathrm{Im}(\F)$, $x$ is in $\F^{n}$, $Y$ in $M_n(\F)$ satisfies $Y+Y^\ast=0$, and the trace of $Y$ must be $-X$ when $\F=\bC$. Here, the star $\ast$ for a matrix is the conjugate transpose. 
   We let
   \[
   K=G\cap O(|\,\,\,|)
   \]
   which is a closed subgroup of $G$ that preserves the canonical Euclidean metric on $\F^{n+1}$. The group $K$ is a maximal compact subgroup of $G$ and it is connected: in fact all elements in $K$ can be written as
   \[
\mathrm{exp}\begin{bmatrix} X & 0 \\ 0 & Y \end{bmatrix}
   \]
where $X$ in $\mathrm{Im}(\F)$ and $Y$ in $M_n(\F)$ satisfies $Y+Y^\ast=0$ (and the trace of $Y$ must be $-X$ when $\F=\bC$). We let
   \[
   A= \{\, a(t)\in G \mid t\in \R \, \}
   \]
which is a closed subgroup of $G$ consisting of elements $a(t)$ of the form
   \[
  a(t) = \begin{bmatrix} c_t & 0 & s_t \\ 0 & 1 & 0 \\ s_t & 0 & c_t \end{bmatrix}
   \] 
   for $t$ in $\R$ where the expression has the $(n-1)\times (n-1)$ identity matrix in the middle entry, and $1$ means the identity matrix, and where $c_t=\cosh t$ and $s_t=\sinh t$ are the hyperbolic cosine and sine respectively. With this coordinate, we shall naturally identify $A$ as the Lie group $\R$. The element $a(t)$ can be written as 
   \[
   U \begin{bmatrix} e^{-t} & 0 & 0 \\ 0 & 1 & 0 \\ 0 & 0 & e^{t} \end{bmatrix} U^{-1}
   \] 
   where 
   \[
   U=U^*=U^{-1} = \begin{bmatrix} -1/\sqrt2 & 0 &  1/\sqrt2 \\ 0 & 1 & 0 \\ 1/\sqrt2 & 0 &  1/\sqrt2 \end{bmatrix}.
   \]
 We let
   \[
   V= \{\, v(x, y)\in G \mid x \in \F^{n-1}, y\in \mathrm{Im}(\F) \,\}, 
   \]
   \[
   N=\{\, n(x, y)\in G \mid x \in \F^{n-1}, y\in \mathrm{Im}(\F) \ \}
   \]
   be closed subgroups of $G$ which consist of elements $v(x, y)$, $n(x, y)$ respectively of the form
   \[
  v(x, y) = U\begin{bmatrix} 1 & -x^\ast & (y-x^*x)/2 \\ 0 & 1 & x  \\ 0 & 0 & 1 \end{bmatrix}U^{-1} = \mathrm{exp} \bigg{(}U \begin{bmatrix} 0 & -x^\ast & y/2 \\ 0 & 0 & x  \\ 0 & 0 & 0 \end{bmatrix}  U^{-1} \bigg{)}, 
  \]
  \[
  n(x, y) = U\begin{bmatrix} 1 & 0 & 0 \\ x & 1 & 0  \\ (y-x^*x)/2  & -x^\ast & 1 \end{bmatrix}U^{-1} = \mathrm{exp} \bigg{(}U \begin{bmatrix} 0 & 0 & 0 \\ x & 0 &  0  \\ y/2 & -x^\ast & 0 \end{bmatrix}  U^{-1} \bigg{)}, 
   \] 
   for $x$ in $\F^{n-1}$ and $y$ in $\mathrm{Im}(\F)$ where each of the expressions has an $(n-1)\times (n-1)$ matrix in the middle entry. 
   
   \begin{remark} We have 
     \[
 \underline{v}(x, y) = U \begin{bmatrix} 0 & -x^\ast & y/2 \\ 0 & 0 & x  \\ 0 & 0 & 0 \end{bmatrix}  U^{-1}  =  \begin{bmatrix} -y/4 & x^\ast/\sqrt2 & -y/4 \\ x/\sqrt2 & 0 & x/\sqrt2  \\ y/4 & -x^*/\sqrt2 & y/4 \end{bmatrix}, 
   \]
   \[
\underline{n}(x, y) = U \begin{bmatrix} 0 & 0 & 0 \\ x & 0 &  0  \\ y/2 & -x^\ast & 0 \end{bmatrix}  U^{-1} =  \underline{v}(-x, -y)^\ast = \begin{bmatrix} -y/4 & -x^\ast/\sqrt2 & y/4 \\ -x/\sqrt2 & 0 & x/\sqrt2  \\ -y/4 & -x^*/\sqrt2 & y/4 \end{bmatrix}.
   \]
We have
 \[
   v(x, y) =  \begin{bmatrix} 1+x^\ast x/4 -y/4 & x^\ast/\sqrt2 & x^\ast x/4 -y/4 \\ x/\sqrt2 & 1 & x/\sqrt2  \\ -x^\ast x/4 + y/4 & -x^*/\sqrt2 & 1- x^\ast x/4 + y/4 \end{bmatrix}, 
 \]
 \[
 n(x, y)= v(-x, -y)^\ast =  \begin{bmatrix} 1+x^\ast x/4 -y/4 & -x^\ast/\sqrt2 & -x^\ast x/4+y/4 \\ -x/\sqrt2 & 1 & x/\sqrt2  \\ x^\ast x/4 - y/4 & -x^*/\sqrt2 & 1- x^\ast x/4 + y/4 \end{bmatrix}.
 \]
Let 
 \[
 w_0=  \begin{bmatrix} -1 & 0 \\ 0  & 1 \end{bmatrix}
 \]
 where $-1$ is a $1\times 1$-matrix and $1$ is an $n\times n$-matrix. We have
 \[
 w_0^2=1, \,\,\, w_0v(x,y)w_0=n(x,y), \,\,\, w_0\underline{v}(x, y)w_0 = \underline{n}(x, y).
 \]
  \end{remark}
   
   The closed subgroup $M$ of $G$ is defined as the centralizer of $A$ in $K$. We have
   \[
   M= \{ \, \begin{bmatrix} q & 0  & 0 \\ 0 & m & 0 \\ 0 & 0 & q \end{bmatrix} \in K \,\}.
   \]
The closed subgroup $MA$ normalizes $N$ and $V$ so $MAN$, $MAV$ are closed subgroups of $G$. We set $P=MAN$. 

We give a standard geometric description of the symmetric space $G/K$ and the homogeneous space $G/P$. The Lie group $G$ naturally acts on the projective space $P(\F^{n+1})$ over $\F$ and an orbit 
\[
G\cdot [1, 0, \cdots, 0]^T
\]
consists of points of the form $[z_0, z_1, \cdots, z_n]^T$ where $\sum_{j=1}^{j=n}|z_j|^2<|z_0|^2$. The isotropy subgroup of $G$ at the point $[1, 0, \cdots, 0]^T$ is the maximal compact subgroup $K$ so this orbit is canonically identified as $G/K$. Let 
\[
d=d_G=\mathrm{dim}_{\R}\F.
\]
Later, we shall use the same notation $d$ for the de-Rham differential operator but it should not cause any confusion. We have the following standard identification:
\[
Z=G/K = \{ (z_j)_{j=1}^{j=n} \in \F^n \mid \sum_{j=1}^{j=n}|z_j|^2<1\} = \mathbb{D}^{dn}
\]
of $G/K$ with the $dn$-dimensional disk $\mathbb{D}^{dn}$ in the real Euclidean space where we identify the point $(z_j)_{j=1}^{j=n}$ in the disk with the point $[1, z_1, \cdots, z_n]^T$ in the projective space. With this identification, the $G$-action on the disk $\mathbb{D}^{dn}=G/K$ can be written as
\[
\begin{bmatrix} a & b \\ c & d \end{bmatrix} \begin{bmatrix} z_1 \\ \vdots \\ z_n\end{bmatrix} =\bigg{(}  c+ d \begin{bmatrix} z_1 \\ \vdots \\ z_n\end{bmatrix} \bigg{)}  \bigg{(}a+ b\begin{bmatrix} z_1 \\ \vdots \\ z_n\end{bmatrix}  \bigg{)}^{-1}\,\,\, \text{for} \,\,\,g=\begin{bmatrix} a & b \\ c & d \end{bmatrix} \in G,
\]
where $a$ is in $\F$, $b$ is a $1\times n$-matrix, $c$ is an $n\times 1$-matrix and $d$ in an $n\times n$ matrix over $\F$. This formula for the $G$-action still makes sense on the boundary sphere $S^{dn-1}$ of $\mathbb{D}^{dn}$. The maximal compact subgroup $K$ acts on the disk and on the sphere as rotations and the isotropy subgroup of $G$ at the point $o=(0, \cdots, 0, 1)$ in the sphere is the closed subgroup $P=MAN$. We note that $K\cap P=M$. We obtain the standard identification
\[
G/P = G\cdot o  = S^{dn-1} = K/M.
\]
Basic computations show 
\[
v(x, y)\cdot o = \begin{bmatrix} \sqrt2x \\ 1-x^*x/2+y/2 \end{bmatrix}  \bigg{(} 1+x^*x/2 - y/2 \bigg{)}^{-1},
\]
\[
n(x, y)\cdot  (-o) =  w_0\cdot (v(x, y)\cdot o) =\begin{bmatrix} -\sqrt2x \\ -1+x^*x/2-y/2 \end{bmatrix}  \bigg{(} 1+x^*x/2 - y/2 \bigg{)}^{-1}
\]
in $S^{dn-1}\subset \F^{n}$ for $x$ in $\F^{n-1}$, $y$ in $\mathrm{Im}(\F)$ and $-o=(0, \cdots, 0, -1)$. Here, the first row has an entry in $\F^{n-1}$ and the second row has an entry in $\F$. From this, we see $V\cap P$ consists only of the identity and $V$ acts transitively on the orbit $V\cdot o$ which is $S^{dn-1}-\{-o\}$. We obtain a Bruhat decomposition 
\[
G/P= VP \sqcup wP \cong V \sqcup \{\infty\}
\]
where $w$ is some fixed element in $K$ of the form
\[
\begin{bmatrix} 1 & 0  & 0 \\ 0 & m & 0 \\ 0 & 0 & -1 \end{bmatrix}.
\]
The Cayley transform for $V$ (and similarly for $N$)
   \begin{equation}\label{eq-Cayley}
   \mathcal{C}\colon V \to G/P=K/M
  \end{equation}
 is defined by 
 \[
 \mathcal{C}(v)=v\cdot o \in G/P
 \]
  or equivalently
  \[
   \mathcal{C}(v)=\tilde K(v)M \in K/M
  \]
  where $g=\tilde K(g)\tilde N(g)\tilde A(g)$ for $g\in G$ with respect to the decomposition $G=KNA$. See \cite{ACD2004} for more on the Cayley transform.

 \subsection*{Analysis on Heisenberg groups}
 Let $\underline{V}$ be a Lie algebra 
 \[
 \underline{V} = \mathfrak{o} \oplus \mathfrak{z}, \,\,\,  \mathfrak{o}=\F^{n-1}, \,\,\, \mathfrak{z}=\mathrm{Im \F} 
 \]
 with Lie Bracket
 \[
 [(x', y'), (x, y)] =\bigg{(} 0,  \,-2\mathrm{Im}(x'^\ast x)\bigg{)}
 \]
 for $x$ in $ \mathfrak{o}$ and $y$ in $\mathfrak{z}$. The Lie algebra $\underline{V}$ is naturally the Lie algebra of the closed subgroup $V$ of $G$ with exponential map
 \[
 \underline{V} \ni (x, y) \mapsto v(x, y) = U\begin{bmatrix} 1 & -x^\ast & (y-x^*x)/2 \\ 0 & 1 & x  \\ 0 & 0 & 1 \end{bmatrix}U^{-1}  \in V.
 \]
 We have $v(x', y')v(x, y)=v(x'+x, y'+y-2\mathrm{Im}(x'^*x))$. The exponential map is a homeomorphism and the Lie group $V$ is a simply connected, two-step nilpotent Lie group (when $\F=\R$, it is abelian). With this coordinate
 \[
 (x, y) \in   \underline{V} = \mathfrak{o} \oplus \mathfrak{z} = \F^{n-1} \oplus \mathrm{Im \F} \cong \R^{d(n-1)}\oplus \R^{d-1}
 \]
 for elements $v(x, y)$ in $V$, we use a Haar measure $d\mu_V=dxdy$ for $V$ where $dx$ and $dy$ are the Lebesgue measures on $\mathfrak{o} =\F^{n-1}$ and on $\mathfrak{z} =\mathrm{Im}(\F)$ which are naturally Euclidean spaces with norm $\lVert\,\,\rVert_{\mathfrak{o}}$  and $\lVert\,\,\rVert_{\mathfrak{z}}$ respectively.
 
 We set
 \[
 \mathcal{N}(x, y)=\bigg{(} |x^*x|^2 + |y|^2 \bigg{)}^{1/4} = \bigg{(} \lVert x\rVert_{\mathfrak{o}}^4 + \lVert y\rVert_{\mathfrak{z}}^2 \bigg{)}^{1/4}.
\]
This is a homogeneous function on the stratified Lie group $V$ of degree one in the sense of Folland \cite[Page 164]{Folland1975}. 

We define  \[
r=r_G=\mathrm{dim}_\R \F^{n-1}+2\mathrm{dim}_\R \mathrm{Im}(\F) = d(n-1)+2(d-1) =d(n+1)-2,
\]
namely
\[
r= \begin{cases} n-1 & \text{for $\SO_0(n,1)$ ($\F=\R$),}  \\  2n &  \text{for $\SU(n,1)$ ($\F=\bC$)}, \\ 4n+2 & \text{for $\Sp(n,1)$ ($\F=\Ha$)}. \end{cases}
\]

 \begin{lemma}(see \cite[Lemma 1.1]{Cowling2010}) \label{lem_locint} For any complex number $\xi$ with $\mathrm{Re}(\xi)>0$, $\mathcal{N}^{\xi-r}$ is locally integrable everywhere on $V$ and defines a distribution 
 \[
 f \mapsto \int_{V} f(x^{-1})\mathcal{N}^{\xi-r}(x)d\mu_V(x)
 \]
 on $C_c^\infty(V)$. If $\mathrm{Re}(\xi)\leq0$, $\mathcal{N}^{\xi-r}$ is  not locally integrable at the origin of $V$. \qed
 \end{lemma}

Given $X$ in $\underline{V}$, we also write $X$ for the associated left-invariant vector field on $V$, i.e.,
\[
Xf(v)= \frac{d}{dt}f(v\mathrm{exp}(tX)) \bigg{|}_{t=0}
\]
for a smooth function $f$ on $V$ and for $v$ in $V$. Fix an orthonormal basis $\{E_j\}_{j=1}^{d(n-1)}$ of $\mathfrak{o}$. 
We define a sub-Laplacian $\Delta_{\mathfrak{o}}$ on $V$ by 
\[
\Delta_{\mathfrak{o}} = -\sum_{j=1}^{d(n-1)}E_j^2.
\]
 Folland (\cite[Section 3]{Folland1975}) showed that the sub-Laplacian $\Delta_{\mathfrak{o}}$ is an essentially selfadjoint, positive definite operator on the Hilbert space $L^2(V, d\mu_V)$ with domain $C^\infty_c(V)$ of compactly supported, smooth functions on $V$ (\cite[Theorem 3.8, Proposition 3.9]{Folland1975}). Thus, the power $\Delta_{\mathfrak{o}}^\xi$ as an unbounded operator on $L^2(V, d\mu_V)$ makes sense for any complex number $\xi$. For any $\alpha\geq0$, $\Delta_{\mathfrak{o}}^\alpha$ is essentially selfadjoint on $C^\infty_c(V)$ (see \cite[Theorem 4.5]{Folland1975}).
\begin{proposition}(\cite[Proposition 2.6]{Cowling2010}) \label{prop_conv} For any complex number $\xi$ with $\mathrm{Re}(\xi)>0$, the composition
 \[
 \Delta_{\mathfrak{o}}^{\xi/2} \circ (\ast \mathcal{N}^{\xi-r}),
 \]
defined as distribution is bounded on $L^2(V, d\mu_V)$. Here, $(\ast \mathcal{N}^{\xi-r})$ is a convolution from the right. In particular, taking $\xi=r/2$,
 \[
 \Delta_{\mathfrak{o}}^{r/4}\circ (\ast \mathcal{N}^{-r/2})
 \]
extends to a bounded operator on $L^2(V, d\mu_V)$. 
 \end{proposition} 
\begin{remark} The composition $\Delta_{\mathfrak{o}}^{\xi/2} \circ (\ast \mathcal{N}^{\xi-r})$ is invertible on $L^2(V, d\mu_V)$ when $\xi$ is not even integer (see \cite[Proposition 2.9]{Cowling2010}).
 \end{remark}
 
 For any real number $\alpha$, we define a homogeneous Sobolev space $\hill^{\alpha}(V)$ to be the completion of $C_c^{\infty}(V)$ with respect to the following norm:
 \[
 \lVert f\rVert_{ \hill^{\alpha}(V)} = \s{\Delta_{\mathfrak{o}}^\alpha f, f}^{1/2}_{L^2(V, d\mu_V)}=\lVert \Delta_{\mathfrak{o}}^{\alpha/2} f\rVert_{L^2(V, d\mu_V)}.
 \]
The multiplication by a compactly supported smooth function is continuous on $\hill^{\alpha}(V)$ for $|\alpha|<r/2$ (see \cite[Theorem 3.6]{ACD2004} for example) but not necessarily for $|\alpha|\geq r/2$.

 We also define a non-homogeneous Sobolev space $\dot{\hill}^{\alpha}(V)$ as in \cite[Section 4]{Folland1975} which is defined as the completion of $C_c^{\infty}(V)$ with respect to the following norm:
 \[
 \lVert f\rVert_{\dot{\hill}^{\alpha}(V)} = \lVert (1+\Delta_{\mathfrak{o}})^{\alpha/2} f\rVert_{L^2(V, d\mu_V)}.
 \]
 The multiplication by a compactly supported smooth function is continuous on $\hill^{\alpha}(V)$ for any $\alpha$ (\cite[Corollary 4.15]{Folland1975}).
 
These two norms are locally equivalent in a sense that for any bounded open region $\Omega$ of $V$, there is a constant $C_\Omega$ such that
 \[
\lVert \Delta_{\mathfrak{o}}^{\alpha/2} f\rVert_{L^2(V, d\mu_V)}  \leq  \lVert (1+\Delta)_{\mathfrak{o}}^{\alpha/2} f\rVert_{L^2(V, d\mu_V)} \leq C_\Omega \lVert \Delta_{\mathfrak{o}}^{\alpha/2} f\rVert_{L^2(V, d\mu_V)}
 \]
 holds for any $f$ in $C^\infty_c(\Omega)$. This follows from the following:
 \begin{lemma}  For any $\alpha>0$, for any bounded open region $\Omega$ of $V$, $\Delta_{\mathfrak{o}}^\alpha$ is bounded away from zero on $C_c^\infty(\Omega)$. That is, three is $C=C(\Omega, \alpha)>0$ such that 
 \[
 \s{f, \Delta_{\mathfrak{o}}^{\alpha}f}_{L^2(V, d\mu_V)} \geq C \s{f,  f}_{L^2(V, d\mu_V)}
 \]
 for any $f$ in $C_c^\infty(\Omega)$.
 \end{lemma}
\begin{proof} By \cite[Corollary 4.16]{Folland1975}, there is $c_1>0$ such that 
\[
\s{f, f}_{L^2(V, d\mu_V)}+\s{f, \Delta_{\mathfrak{o}}^{2}f}_{L^2(V, d\mu_V)}\geq c_1\s{f, \Delta f}_{L^2(V, d\mu_V)}
\]
for any $f \in C_c^\infty(\Omega)$ where $\Delta$ is the Euclidean Laplacian on $V$. Using the ellipticity of $\Delta$ and the relative compactness of $\Omega$, we see that there is $c_2$ such that
\[
\s{f, \Delta_{\mathfrak{o}}^{2}f}_{L^2(V, d\mu_V)} \geq c_2\s{f,  f}_{L^2(V, d\mu_V)}.
\] 
It follows for any integer $N>0$, $\s{f, \Delta_{\mathfrak{o}}^{2N}f}_{L^2(V, d\mu_V)} \geq c_2^N\s{f,  f}_{L^2(V, d\mu_V)}$. The general case follows by interpolation.
\end{proof}
  
From Proposition \ref{prop_conv}, we see that the convolution $\ast\mathcal{N}^{-r/2}$ on $C_c^\infty(V)$ extends to a bounded operator from $L^2(V, d\mu_V)$ to $\hill^{r/2}(V)$.
 
Consider the evaluation map 
\[
\mathrm{ev}_0\colon C^\infty_c(V) \to \bC
\]
 at the origin of $V$, which we regard as an unbounded functional (operator) on Sobolev spaces $\hill^{\alpha}(V)$ and $\dot{\hill}^{\alpha}(V)$,
 \begin{lemma} \label{lem_notcont} The functional $\mathrm{ev}_0$ is not bounded on $\hill^{r/2}(V)$.
 \end{lemma}
 \begin{proof}
Suppose for the contradiction, the evaluation $\mathrm{ev}_0$ is bounded on $\hill^{r/2}(V)$. Then, the composition 
 \[
 \mathrm{ev}_0\circ (\ast\mathcal{N}^{-\frac{r}{2}})\colon L^2(V, d\mu_V) \to \hill^{\frac{r}{2}}(V) \to \bC,
 \]
 which is nothing but the distribution $\mathcal{N}^{-\frac{r}{2}}$ on $V$ (acting from right), would be bounded on $L^2(V, d\mu_V)$. On the other hand, the function $\mathcal{N}^{-\frac{r}{2}}$ is not locally square-integrable on $V$ at the origin (see Lemma \ref{lem_locint}), a contradiction. 
  \end{proof}
  
   \begin{lemma} \label{lem_notcont2} The functional $\mathrm{ev}_0$ is not bounded on $\dot{\hill}^{r/2}(V)$.
 \end{lemma}
 \begin{proof}
 This does not directly follow from Lemma \ref{lem_notcont} since the multiplication by a non-zero compactly supported smooth function is not necessarily continuous on $\hill^{\alpha}(V)$ for $\alpha\geq r/2$. On the other hand, the multiplication by a non-zero compactly supported smooth function is continuous on $\dot{\hill}^{\alpha}(V)$ for any $\alpha$. Therefore, the assertion is equivalent to that the functional $\mathrm{ev}_0$ is not bounded on $C_c^\infty(\Omega_0)$ for some fixed open neighborhood $\Omega$ of the origin of $V$ with respect to the non-homogeneous Sobolev norm for $\dot{\hill}^{r/2}(V)$. Suppose for the contradiction, the evaluation $\mathrm{ev}_0$ is bounded on $C_c^\infty(\Omega_0)$ with respect to the non-homogeneous Sobolev norm. Then, it is bounded on $C_c^\infty(\Omega_0)$ with respect to the homogeneous Sobolev norm as well.  Now, the homogeneous Sobolev norm is invariant under the homogeneous dilation $\alpha_t\colon (x, y)\mapsto (tx, t^2y)$ on $V$. It follows that the evaluation $\mathrm{ev}_0$ is bounded on $C_c^\infty(\alpha_t(\Omega_0))$ for all $t>0$ with the same constant for the bound. It follows that $\mathrm{ev}_0$ is bounded on the entire space ${\hill^{r/2}}(V)$, a contradiction to Lemma \ref{lem_notcont}. Thus, $\mathrm{ev}_0$ is not bounded on $\dot{\hill}^{r/2}(V)$ as well.
  \end{proof}

\begin{remark} When $G=\SO_0(n ,1)$, Lemma \ref{lem_notcont2} is same as saying that the evaluation map on the Sobolev space $W^{n/2, 2}(\R^n)$ is not continuous, which is well known and which can be easily checked using elementary Fourier analysis for example. This is one of the critical cases of the Sobolev embedding theorem. 
\end{remark}


 \subsection*{Principal series representations}
 We use the normalized Haar measure $d\mu_K$, $d\mu_M$ on $K$ and on $M$ respectively and the standard Lebesgue measures $d\mu_A$, $d\mu_N$, $d\mu_V$ on $A$, on $N$ and on $V$ respectively. Here, $d\mu_N$ is defined analogously to $d\mu_V$. The following formula defines a Haar measure $d\mu_G$ on $G$ (all of these groups are unimodular so these measures are left and right invariant):
\[
\int_G f(g) d\mu_G = \int_K d\mu_K(k) \int_N d\mu_N(n) \int_A d\mu_A(a) f(kna).
\]
We have (see \cite[Page 88]{Cowling2010})
\[
\int_G f(g) d\mu_G = C_G  \int_V d\mu_V(v) \int_N d\mu_N(n) \int_A d\mu_A(a)   \int_M d\mu_M(m)   f(vnam)
\]
for some positive constant $C_G$ which depends only on $G$.

\vspace{10pt}

Recall $P=MAN$ is a closed subgroup of $G$ and $N$ is a closed normal subgroup of $P$. We define a (non-unitary) character $\rho$ on $A$ as
\[
\rho(a(t))=\mathrm{exp}(rt/2).
\]
\begin{remark} This is the exponential of the half-sum of the roots of $\mathfrak{a}$ on $\mathfrak{n}$. It is the square-root of the Jacobian of the conjugation action $n\mapsto ana^{-1}$ of $A$ on $N$ so we have
\[
\int_N f(ana^{-1}) \rho(a)^{2}  d\mu_N(n) =\int_N f(n)d\mu_N(n).
\]
We have (see \cite[Page 90-91]{Cowling2010})
\[
\int_N f(a^{-1}na)  \rho(a)^{-2} d\mu_N(n) =\int_N f(n)d\mu_N(n),
\]
\[
\int_V f(a^{-1}va)\rho(a)^{2}d\mu_V(v) =  \int_V f(v)d\mu_V(v).
\]
\end{remark}
For any unitary irreducible (finite-dimensional) representation $\mu$ of $M$ on $\hill_\mu$ and for any complex number $\lambda$, we consider the vector space $I_{\mu, \lambda}$\footnote{For  $\mathcal{I}_{\mu, \lambda}$ from the previous section, we have $\mathcal{I}_{\mu, \lambda}=I_{\mu, r\lambda}$.} of $\hill_\mu$-valued measurable functions $f$ on $G$ satisfying
\[
f(gman)=\mu^{-1}(m)\mathrm{exp}(-(r+\lambda)t/2) f(g) = \mu^{-1}(m)\mathrm{exp}(-\lambda t/2)\rho^{-1}(a) f(g)
\]
for any $man$ in $MAN=P$ where $a=a(t)$. The group $G$ acts on functions in $I_{\mu, \lambda}$ by the left-translation.

\begin{remark} Take $\mathrm{Re}(\lambda)=0$. We have
\[
\int_V \lVert f(a^{-1}v)\rVert^2d\mu_V(v) = \int_V \lVert f(a^{-1}vaa^{-1})\rVert^2d\mu_V(v) 
\]
\[
= \int_V \lVert f(a^{-1}va)\rVert^2 \rho(a)^2 d\mu_V(v)  = \int_V \lVert f(v)\rVert^2d\mu_V(v).
\]
\end{remark}

We define
\[
\lVert f\rVert_p^{(V)} = \bigg{(}\int_V \lVert f(v)\rVert^pd\mu_V(v)\bigg{)}^{1/p},
\]
\[
 \lVert f\rVert_p^{(K)} = \bigg{(}\int_K \lVert f(k)\rVert^pd\mu_K(k)\bigg{)}^{1/p} = \bigg{(}\int_{K/M} \lVert f(kM)\rVert^pd\mu_{K/M}(kM)\bigg{)}^{1/p} = \lVert f\rVert_p^{(K/M)}.
\]

\begin{proposition}\label{prop_lp}(\cite[Lemma 5.2]{Cowling2010}) Suppose $\lambda$ is in the tube $T$ given by
\[
T = \{ \lambda \in \bC \mid \mathrm{Re}(\lambda) \in [-r, r] \}
\]
and let $p\in [1, +\infty]$ be given by the formula
\[
1/p= \mathrm{Re}(\lambda)/2r + 1/2.
\]
Then, for any measurable function $f$ in $I_{\mu, \lambda}$ we have
\[
\int_K\lVert f(k)\rVert^pd\mu_K(k) = C_G \int_V \lVert f(v)\rVert^pd\mu_V(v)
\]
and $G$ acts isometrically on $I_{\mu, \lambda}^{L^p}$, the space of functions $f$ in $I_{\mu, \lambda}$ for which $\lVert f\rVert_p^{(V)}$ is finite, equipped with the same norm. If $p=\infty$, the integral is replaced by the essential supremum. 
\end{proposition}
 
   \subsection*{Uniformly bounded representations (non-compact picture)}
We let $I^{\infty}_{\mu, \lambda}$ be the subspace of $I_{\mu, \lambda}$ consisting of smooth functions on $G$. Let
\[
L^2(V; \hill_\mu)=L^2(V, d\mu_V)\otimes \hill_\mu,  \,\, H^\alpha(V; \hill_\mu)=H^\alpha(V)\otimes \hill_\mu.
\]
   
 For any real number $\alpha$, we let $I^{\hill^\alpha(V; \hill_\mu)}_{\mu, \lambda}$ be the closure of the subspace of $I^{\infty}_{\mu, \lambda}$ of those functions $f$ whose restriction $f^{(V)}$ to $V$ have compact support in $V$, with respect to the norm
   \[
   \lVert f\rVert_{I^{\hill^\alpha(V; \hill_\mu)}_{\mu, \lambda}} =  \lVert f^{(V)}\rVert_{\hill^\alpha(V; \hill_\mu)} = \lVert\Delta_{\mathfrak{o}}^{\alpha/2} f^{(V)}\rVert_{L^2(V; \hill_\mu)}.
   \]
   \begin{theorem}(\cite[Theorem 7.1]{Cowling2010})\label{thm-Cowling-noncompact} Suppose $\lambda$ is inside the tube $T$, namely suppose $\mathrm{Re}(\lambda) \in (-r, r)$. Then, $G$ acts uniformly boundedly on the Hilbert space  $I^{\hill^\alpha(V)}_{\mu, \lambda}$ for $\alpha=-\mathrm{Re}(\lambda)/2$.  
   \end{theorem}
   
   \begin{remark}
   As explained in \cite[Page 112]{Cowling2010}, it follows from this theorem that $I^{\infty}_{\mu, \lambda}\subset I^{\hill^\alpha(V; \hill_\mu)}_{\mu, \lambda}$.
   \end{remark}
   
    \subsection*{Quasi-conformal structure on $G/P=K/M$}

We recall from the discussion after Corollary \ref{cor-TFAE} that there is a $G$-equivariant subbundle $E$ of the tangent bundle $T(G/P)$ of codimension $1$ for $G=\SU(n, 1)$ and of codimension $3$ for $G=\Sp(n, 1)$.  We set $E=T(G/P)$ if $G=\SO_0(n, 1)$. 

   \subsection*{Uniformly bounded representations (compact picture)}
   
  Recall $I^{\infty}_{\mu, \lambda}$ is the subspace of $I_{\mu, \lambda}$ consisting of smooth functions on $G$. The space $I^{\infty}_{\mu, \lambda}$ is naturally identified with the space of $\hill_\mu$-valued smooth functions $f$ on $K$ satisfying
  \[
  f(km)=\mu(m)^{-1}f(k).
  \]
The latter space is naturally identified with the space $\Gamma(K/M; E_\mu)$ of smooth sections of the associated vector bundle
  \[
  E_\mu = K\times_M \hill_\mu = \{\,[k, v] \mid k\in K, \,\, v\in \hill_\mu   \, \}
  \]
  on the sphere $K/M=G/P$ where $[k, v]=[km, \mu(m)^{-1}v]$. This is the compact picture of the principal series representations. The space $\Gamma(K/M; E_\mu)$ of smooth sections of $E_\mu$ is equipped with a natural $L^2$-norm using the $K$-invariant metric on $K/M$. Let $L^2(K/M; E_\mu)$ be its $L^2$-completion.
  
  For any $f$ in $I^\infty_{\mu, \lambda}$, let us denote by $f^{(K/M)}$ in $\Gamma(K/M; E_\mu)$, the corresponding smooth section of the bundle $E_{\mu}$ on $K/M$.
    
Let $\nabla$ be any $K$-invariant connection to the bundle $E_{\mu}$. We define 
\[
\nabla_E\colon \Gamma(K/M; E_\mu) \to \Gamma(K/M; E_\mu\otimes E^\ast)
\]
to be the restriction of the connection to the subbundle $E$ of $T(G/P)$. We think of it as an unbounded operator from $L^2(K/M; E_\mu)$ to $L^2(K/M; E_\mu\otimes E^\ast)$. We define 
 \[
\Delta_E = \nabla_{E}^\ast \nabla_E
 \]
 acting on $\Gamma(K/M; E_\mu)$.
 
 For any real number $\alpha$, we set $I^{\hill^\alpha(K/M; E_\mu)}_{\mu, \lambda}$ to be the closure of $I^{\infty}_{\mu, \lambda}$ with respect to the norm
   \[
   \lVert f\rVert_{I^{\hill^\alpha(K/M; E_\mu)}_{\mu, \lambda}} =  \lVert f^{(K/M)}\rVert_{\hill^\alpha(K/M; E_\mu)} = \lVert(1+\Delta_{E})^{\alpha/2} f^{(K/M)}\rVert_{L^2(K/M; E_\mu)}.
   \]
   
   \begin{lemma}\label{lem-local-equivalence} Let $\alpha\geq0$ and $\Omega$ be a bounded open region of $V$. We let $I^{C^\infty_c(\Omega)}_{\mu, \lambda}$ be the subspace of $I^{\infty}_{\mu, \lambda}$ of those functions $f$ whose restriction $f^{(V)}$ to $V$ have compact support inside $\Omega$. Then, the two norms  $\lVert\,\,\rVert_{I^{\hill^\alpha(V)}_{\mu, \lambda}}$ and $\lVert\,\,\rVert_{I^{\hill^\alpha(K/M)}_{\mu, \lambda}}$ are equivalent on $I^{C^\infty_c(\Omega)}_{\mu, \lambda}$. 
  \end{lemma}
  \begin{proof} We give a proof for $\alpha\in \N$; since we only use these cases. On $I^{C^\infty_c(\Omega)}_{\mu, \lambda}$, both of these two norms are equivalent to the non-homogeneous norm of $\dot{\hill}^{\alpha}(V)$. In fact, via the Cayley transform $\mathcal{C}$ in \eqref{eq-Cayley} and by \cite[Lemma 2.5]{ACD2004}, the norm $\lVert f \rVert_{I^{\hill^\alpha(K/M)}_{\mu, \lambda}}$ on $I^{C^\infty_c(\Omega)}_{\mu, \lambda}$ is equivalent to a norm of the form
  \[
    \lVert  f^{(V)}\rVert_{L^2(V; \hill_\mu)} +  \lVert (\Delta'_{\mathfrak{o}})^{\alpha/2} f^{(V)}\rVert_{L^2(V; \hill_\mu)}
\]
where $ \Delta'_{\mathfrak{o}}=B\Delta_{\mathfrak{o}} + A$ where $B$ is a smooth positive function on $V$ bounded away from $0$ on $\Omega$ and $A$ is an order $1$ operator on $V$ in a weighted sense as in \cite{Folland1975}. For $\alpha=N>0$ integer, then this norm is equivalent to a norm of the form
\[
\s{f^{(V)}, f^{(V)}}_{L^2(V; \hill_\mu)} +  \s{f^{(V)}, (\Delta_{\mathfrak{o}}^{N} + A_N) f^{(V)}}_{L^2(V; \hill_\mu)}
    \]
    for some compactly supported differential operator of weighted order $<2N$. It is not hard to see such a norm is equivalent to the non-homogeneous Sobolev norm
    \[
    \s{f^{(V)}, f^{(V)}}_{L^2(V; \hill_\mu)} +  \s{f^{(V)}, \Delta_{\mathfrak{o}}^{N} f^{(V)}}_{L^2(V; \hill_\mu)}
    \]
    on  $I^{C^\infty_c(\Omega)}_{\mu, \lambda}$.
  \end{proof}
   
   \begin{theorem}(\cite[Theorem 4.2]{ACD2004}) Suppose $\lambda$ is inside the tube $T$, namely suppose $\mathrm{Re}(\lambda) \in (-r, r)$. Then for $\alpha=\mathrm{Re}(\lambda)/2$, the identity on $I^\infty_{\mu, \lambda}$ extends to an isomorphism of Banach spaces
   \[
   I^{\hill^\alpha(V)}_{\mu, \lambda} \cong I^{\hill^\alpha(K/M)}_{\mu, \lambda}.
   \]
   In other words, the two norms  $\lVert\,\,\rVert_{I^{\hill^\alpha(V)}_{\mu, \lambda}}$ and $\lVert\,\,\rVert_{I^{\hill^\alpha(K/M)}_{\mu, \lambda}}$ on $I^\infty_{\mu, \lambda}$ are equivalent.   \qed
   \end{theorem}
   \begin{proof} This is what is proven in \cite[Theorem 4.2]{ACD2004} independently of Cowling's result (Theorem \ref{thm-Cowling-noncompact}) at least for $\mu=1_M$ but the general case can be proven as a simple consequence of Theorem \ref{thm-Cowling-noncompact} as follows. Again, we only give a proof for $\alpha\in \N$. By Lemma \ref{lem-local-equivalence}, the two norms are equivalent on the subspace $C_c^\infty(\Omega_0)$ of $I^\infty_{\mu, \lambda}$ for some neighborhood $\Omega_0$ of the origin of $V$, which is identified as a neighborhood of the point $-o$ in $G/P$ via the Cayley transform $\mathcal{C}$ \eqref{eq-Cayley}. Using Theorem \ref{thm-Cowling-noncompact}, we know that the norm $\lVert\,\,\rVert_{I^{\hill^\alpha(V)}_{\mu, \lambda}}$ is invariant up to equivalences under the the action of $G$, in particular under the action of $K$. We also know that the norm  $\lVert\,\,\rVert_{I^{\hill^\alpha(K/M)}_{\mu, \lambda}}$ is $K$-invariant. Using these, both norms are locally equivalent. On the other hand, multiplication by the smooth functions on $G/P$ is bounded on $I^\infty_{\mu, \lambda}$ with respect to both norms. Here, $|\alpha|<r/2$ is crucial for it to be continuous on $\hill^\alpha(V)$. Using a smooth partition of unity on $G/P$, we see that the two norms are equivalent on the whole space $I^\infty_{\mu, \lambda}$. 
   \end{proof}
   
      \begin{corollary}\label{cor_compact}(\cite{Cowling2010}, \cite{ACD2004}, see also \cite[Theorem 26, Corollary 27]{Julg2019}) Suppose $\lambda$ is inside the tube $T$, namely suppose $\mathrm{Re}(\lambda) \in (-r, r)$. Then, $G$ acts uniformly boundedly on the Hilbert space  $I^{\hill^\alpha(K/M)}_{\mu, \lambda}$ for $\alpha=-\mathrm{Re}(\lambda)/2$.  \qed
   \end{corollary}

 \section{An alternative proof of Shalom's conjecture}\label{sec-exp-proof}
 \subsection*{Cocycle}
 Let $W_0=\Omega^{\mathrm{top}}_{\int=0}(G/P)$ be the complexified vector space of top-degree forms with zero integral on $G/P$ equipped with the natural representation $\pi$ of $G$. Following Julg's construction (\cite[Section 1.4]{Julg1998} \cite[Chapter 3]{CCJJV2001}), we define for $x, y$ in $Z=G/K$,
 \[
 c(x, y)=\mu_y - \mu_x \in W_0=\Omega^{\mathrm{top}}_{\int=0}(G/P)
 \]
 where $\mu_x$ is the visual measure, with respect to $x$, on the boundary sphere $G/P=\partial Z$, i.e. $\mu_x$ is the $K_x$-invariant normalized volume form on the boundary sphere where $K_x$ is the isotropy subgroup of $G$ at $x$. If we naturally identify $Z=G/K$ as the unit disk $\mathbb{D}^{dn}$ in $\R^{dn}$ with origin $0$ and $G/P$ as the standard sphere $S^{dn-1}$ we have
 \[
\mu_0=\mathrm{Vol}_{S^{dn-1}}, \mu_{g0}=(g^{-1})^\ast(\mu_0) 
\]
where $\mathrm{Vol}_{S^{dn-1}}$ is the standard normalized volume form on $S^{dn-1}$. It is clear that the map $c$ is a $G$-equivariant cocycle in a sense that 
\[
c(x, y) + c(y, z) = c(x, z), \,\,\, \text{and}\,\,\, c(gx, gy)=\pi(g)c(x, y).
\]

We recall here a result of Pierre Julg:
\begin{theorem}(\cite[Section 1.4]{Julg1998}, \cite[Chapter 3]{CCJJV2001}) \label{thm-Julg} Let $G$ be any one of the simple real-rank-one Lie groups $\SO_0(n, 1)$, $\SU(n,1)$, $\Sp(n,1)$, $F_{4(-20)}$. There is a $G$-invariant quadratic form $Q$ on $W_0$ for which the cocycle $c$ is proper in a sense that
\[
Q(c(x,y)) \to +\infty \,\,\, \text{as} \,\,\, d_Z(x,y)\to +\infty.
\]
where $d_Z$ is the distance function on the symmetric space $Z=G/K$. Furthermore, if $G$ is either $\SO_0(n, 1)$ or $\SU(n,1)$, $Q$ is positive definite on $W_0$. With respect to the topology on $W_0$ induced by the quadratic form $Q$, $c$ is continuous. It follows that the groups $\SO_0(n, 1)$ and $\SU(n,1)$ are a-T-menable (or equivalently have the Haagerup property).
\end{theorem}

We recall how we deduce the last part. Given a $G$-equivariant cocycle $c\colon Z\times Z\to W_0$ as above, a group $1$-cocycle $b_g$ in $(\pi, W_0)$ can be constructed as $b_g=c(gx_0, x_0)$ for any chosen point $x_0$ in $Z$. The properness of the cocycle $c$ implies that (is equivalent to) the properness of $b_g$ with respect to the norm induced by $Q$. 

We note that the quadratic form $Q$ cannot be positive definite for $\Sp(n, 1)$ for $n\geq2$ and $F_{4(-20)}$ because these groups have Kazhdan's property (T). 

We shall reprove Theorem \ref{thm-Shalom-conj} in the following form:

\begin{theorem}\label{thm-Julg-sp}  Let $G$ be any one of $\SO_0(n, 1)$ $n\geq2$, $\SU(n,1)$ for $n\geq2$, and $\Sp(n,1)$ for $n\geq2$. Let $\pi$ be the natural representation of $G$ on $W_0=\Omega^{\mathrm{top}}_{\int=0}(G/P)$. Then, there is a Euclidean norm $\lVert\,\,\rVert_{W_0}$ on $W_0$ for which $\pi$ is continuous and uniformly bounded and for which the cocycle $c$ is proper in sense that
\[
\norm{c(x,y)}_{W_0} \to +\infty \,\,\, \text{as} \,\,\, d_Z(x,y)\to +\infty.
\]
With respect to the topology on $W_0$ induced by the norm $\lVert\,\,\rVert_{W_0}$, $c$ is continuous. It follows that $G$ admits a proper $1$-cocycle in a uniformly bounded representation on a Hilbert space. In particular, if $\Gamma$ is any non-compact closed subgroup inside the Lie groups $\SO_0(n, 1)$ ($n\geq2$), $\SU(n,1)$ ($n\geq2$), $\Sp(n,1)$ ($n\geq2$) and their products, $\Gamma$ admits a proper $1$-cocycle in a uniformly bounded representation on a Hilbert space.
\end{theorem}

 \subsection*{Euclidean norm on $W_0$}
 We recall from Theorem \ref{thm-W0-ub}, that there is a Euclidean norm on $W_0=\Omega^{\mathrm{top}}_{\int=0}(G/P)$ for which the natural representation $\pi$ of $G$ is continuous and uniformly bounded and that it was constructed as the dual norm on a suitable Euclidean norm on the quotient space $\Omega^0(G/P)/\bC1_{G/P}$. The norm on $\Omega^0(G/P)/\bC1_{G/P}$ is defined by embedding it to $\Gamma(E^\ast)$ through the de-Rham differential $d_E$ where the latter space was naturally the space of the induced representation $\pi_{\mu, \lambda}$ for $\lambda$ inside the tube $T$ so that the result of Cowling is applicable. We shall describe these norms on $\Omega^0(G/P)/\bC1_{G/P}$ explicitly so that we can directly show the properness of the cocycle $c$ with respect to their dual norms on $W_0$.

  \subsection*{The case of $\SO_0(n ,1)$}
 We first discuss the case when $G$ is $\SO_0(n ,1)$ since it is simple. Via the de-Rham differential $d\colon C^\infty(G/P) \to \Omega^1(G/P)$, the quotient representation $\pi_0$ on $\Omega^0(G/P)/\bC1_{G/P}$ can be identified as a sub-representation of natural representation $\pi_1$ of $G$ on $\Omega^1(G/P)$. 

 \begin{theorem}\label{thm_so}(\cite{Cowling2010}, \cite{ACD2004}, \cite[Corollary 27]{Julg2019}) Let $G=\SO_0(n ,1)$ for $n\geq3$. Define an inner product $\s{\,\,,\,\, }$ and its associated norm $\lVert\,\,\rVert$ on $\Omega^1(G/P)$ by 
 \[
 \s{w_1, w_2}= \s{w_1, \Delta^{\frac{n-1}{2}-1}w_2}_{\Omega^1_{L^2}(G/P)}, \lVert w\rVert=\s{w, w}^{\frac12}=\lVert\Delta^{\frac{n-1}{4}-\frac12}w\rVert_{\Omega^1_{L^2}(G/P)}
 \]
 where the Laplacian $\Delta$, the inner product $\s{\,\,,\,\,}_{\Omega^1_{L^2}(G/P)}$ and the norm $\lVert\,\,\rVert_{\Omega^1_{L^2}(G/P)}$ are the ones defined by the standard Riemannian metric on $G/P=S^{n-1}$. Then, the natural representation $\pi_1$ of $G$ on $\Omega^1(G/P)$ is continuous and uniformly bounded with respect to the norm $\lVert\,\,\rVert$.
 \end{theorem}
 \begin{proof} As we saw in Lemma \ref{lem-principal}, the representation $\pi_1$ on $\Omega^1(G/P)$ is naturally identified as the principal series representation $\pi_{\mu, \lambda}$ on $I^\infty_{\mu, \lambda}$ for some finite-dimensional unitary representation $\mu$ of $M$ and for $\lambda=-r+2 \in (-r, r)$ (note that we have $\mathrm{I}^\infty_{\mu, \lambda}=I^\infty_{\mu, r\lambda}$) where $r=n-1$. The claim follows from Corollary \ref{cor_compact} and from the fact that the Laplacian is bounded away from zero on $\Omega^\ast(G/P)$ except for the zeroth and the top-degree forms.
 \end{proof}
By restricting our attention to the image of $C^\infty(G/P)/\bC1_{G/P}$ inside $\Omega^1(G/P)$, we obtain the following:

 \begin{corollary}\label{cor-so} Let $G=\SO_0(n ,1)$ for $n\geq2$. Define an inner product $\s{\,\,,\,\, }$ and its associated norm $\lVert\,\,\rVert$ on $C^\infty(G/P)/\bC1_{G/P}$ by 
 \[
 \s{\phi, \phi}= \s{d\phi, \Delta^{\frac{n-1}{2}-1}d\phi}_{\Omega^1_{L^2}(G/P)}=\s{\phi, \Delta^{\frac{n-1}{2}}\phi}_{L^2(G/P)},
 \]
 \[
 \lVert\phi\rVert=\s{\phi, \phi}^{\frac12}=\lVert\Delta^{\frac{n-1}{4}}\phi\rVert_{L^2(G/P)}
 \]
 where the inner product $\s{\,\,,\,\,}_{L^2(G/P)}$ and the norm $\lVert\,\rVert_{L^2(G/P)}$ are the ones defined by the standard normalized volume form on $G/P=S^{n-1}$. Then, the representation $\pi_0$ of $G$ on $C^\infty(G/P)/\bC1_{G/P}$ is uniformly bounded with respect to the norm $\lVert\,\,\rVert$. 
 \end{corollary}
 \begin{proof} When $n=2$, via the composition of the Poisson transform and the de-Rham differential $d$, $\pi_0$ is naturally identified as the square integrable representation on  $L^2$-harmonic one-forms on $G/K$ which is a unitary representation of $G$. A direct computation shows that the given norm on $C^\infty(G/P)/\bC1_{G/P}$ coincides, up to scalar, with the one on $L^2$-harmonic one-forms via this identification. For $n\geq3$, the claim follows from Theorem \ref{thm_so}.
 \end{proof}
We remark that these uniformly bounded representation are unitarizable but not unitary unless $n=2$ or $3$. The ``correct'' norm involves more complicated functional calculus of the Laplacian $\Delta$. Note that equipped with this norm, $C^\infty(G/P)/\bC1_{G/P}$ is more or less (the quotient of) the Sobolev space $W^{\frac{n-1}2, 2}(G/P)$ which is the one appearing in the critical case of the Sobolev embedding: for $\epsilon>0$, we have the following natural continuous embedding:
 \[
 W^{\frac{n-1}{2}+\epsilon, 2}(G/P) \to C(G/P),
 \]
 but this fails to be well-defined (continuous) at $\epsilon=0$. This will be essentially the reason for the properness of the cocycle $c$.
 

  \subsection*{The case of $\SU(n ,1)$ and $\Sp(n,1)$}
Recall that $E$ is a $G$-equivariant subbundle of $T(G/P)$ of codimension $1$ when $G=\SU(n, 1)$ and of codimension $3$ when $G=\Sp(n, 1)$. We define $\Gamma(E^\ast)$ to be the complexified vector space of smooth sections of the bundle $E^\ast$ and 
\[
d_E=\,\,\, \mid_E \circ d \colon C^\infty(G/P) \to \Omega^1(G/P) \to \Gamma(E^\ast)
\]
to be the composition of the de-Rham differential $d$ and the restriction of one-forms defined on $T(G/P)$ to the subbundle $E$. With respect to the natural representations of $G$ on $C^\infty(G/P)$ and on $\Gamma(E^\ast)$, $d_E$ is $G$-equivariant. The kernel of $d_E$ is spanned by the constant function $1_{G/P}$. Thus, we can regard the representation $\pi_0$ on the quotient space $C^\infty(G/P)/\bC1_{G/P}$ as a sub-representation of the natural representation $\pi_{E^\ast}$ of $G$ on $\Gamma(E^\ast)$. We define a sub-Laplacian $\Delta_E=\nabla_E^\ast\nabla_E$ on $\Gamma(E^\ast)$ as before choosing a $K$-invariant connection $\nabla$ on $E^\ast$.

 \begin{theorem}\label{thm_sp}(\cite{Cowling2010}, \cite{ACD2004}, \cite[Corollary 27]{Julg2019})  Let $G=\SU(n ,1)$ for $n\geq2$ and $\Sp(n, 1)$ for $n\geq2$. Define an inner product $\s{\,\,,\,\,}$ and its associated norm $\lVert\,\,\rVert$ on $\Gamma(E^\ast)$ by 
 \[
 \s{w_1, w_2}= \s{w_1, (1+\Delta_E)^{\frac{r}{2}-1}w_2}_{\Gamma_{L^2}(E^\ast)},
 \]
 \[
  \lVert w\rVert=\s{w, w}^{\frac12}=\lVert(1+\Delta_E)^{\frac{r}{4}-\frac12}w\rVert_{\Gamma_{L^2}(E^\ast)}
 \]
 where the inner product $\s{\,\,,\,\,}_{\Gamma_{L^2}(E^\ast)}$ and the norm $\lVert\,\,\rVert_{\Gamma_{L^2}(E^\ast)}$ are the ones defined by the $K$-invariant metric on $G/P$. Then,  the natural representation $\pi_{E^\ast}$ of $G$ on $\Gamma(E^\ast)$ is continuous and uniformly bounded with respect to the norm $\lVert\,\,\rVert$.
 \end{theorem}
 \begin{proof} As we saw in Lemma \ref{lem-principal}, the representation $\pi_{E^\ast}$ of $G$ on $\Gamma(E^\ast)$ is naturally identified as the principal series representation $\pi_{\mu, \lambda}$ on $I^\infty_{\mu, \lambda}$ for some finite-dimensional unitary representation $\mu$ of $M$ and for $\lambda=-r+2 \in (-r, r)$. The claim follows from Corollary \ref{cor_compact}.
 \end{proof}
By restricting our attention to the image of $C^\infty(G/P)/\bC1_{G/P}$ inside $\Gamma(E^\ast)$, we obtain the following:

 \begin{corollary}\label{cor-sp}  Let $G=\SU(n ,1)$ for $n\geq2$ and $\Sp(n, 1)$ for $n\geq2$. Define an inner product $\s{\,\,,\,\, }$ and its associated norm $\lVert\,\,\rVert$ on $C^\infty(G/P)/\bC1_{G/P}$ by 
 \[
 \s{\phi, \phi}= \s{d_E\phi, (1+\Delta_E)^{\frac{r}{2}-1}d_E\phi}_{\Gamma_{L^2}(E^\ast)},
 \]
 \[
 \lVert\phi\rVert=\s{\phi, \phi}^{\frac12}=\lVert(1+\Delta_E)^{\frac{r}{4}-\frac12}d_E\phi\rVert_{\Gamma_{L^2}(E^\ast)}.
 \]
Then, the natural representation $\pi_0$ of $G$ on $C^\infty(G/P)/\bC1_{G/P}$ is continuous and uniformly bounded with respect to the norm $\lVert\,\,\rVert$. \qed
 \end{corollary}

\begin{proof}[~Proof of Theorem \ref{thm-Julg-sp}]

We equip $W_0=\Omega^{\mathrm{top}}_{\int=0}(G/P)$ with the norm $\norm{\,}_{W_0}$ which is the dual to the one on $C^\infty(G/P)/\bC1_{G/P}$ as in Corollary  \ref{cor-so}, \ref{cor-sp}. We know that the representation $\pi$ is continuous and uniformly bounded with respect to the norm $\norm{\,}_{W_0}$. The continuity of the cocycle $c$ with respect to the norm $\norm{\,}_{W_0}$ follows from the continuity of $\pi$. It remains to show the following:
 
\begin{lemma} The cocycle $c(x,y)=\mu_y - \mu_x$ in $W_0$ is proper with respect to $\lVert\,\,\rVert_{W_0}$ in a sense that
\[
\lVert c(x, y)\rVert_{W_0} \to +\infty \,\,\, \text{as} \,\,\, d_Z(x,y)\to +\infty.
\]
\end{lemma}
\begin{proof}
By double transitivity, for pairs of points with same distance, of the $G$-action on $Z=G/K$ and by the uniform-boundedness of $\pi$ on $W_0$, it suffices to show that 
\begin{equation}\label{proper}
\lim_{x\to o} =\lVert\mu_x - \mu_0\rVert_{W_0} = +\infty
 \end{equation}
for $o=(1, 0, \cdots, 0)$ in $\partial Z=G/P=S^{dn-1}$. For this, it is enough to show that there are a small open neighborhood $U_o$ of $o$ in $G/P$ and a sequence $\phi_n$ of functions in $C^\infty_c(U_o)$ such that 
 \begin{enumerate}
 \item $\lVert\phi_n\rVert_{\hill^{r/2}(K/M)}$ are uniformly bounded for $n\geq1$,
 \item $\lim_{n\to \infty} \phi_n(o)=+\infty$.
 \end{enumerate}
 
 Indeed, if this is the case, we may translate $(-1)\phi_n$ to around the antipodal $-o$ of $o$ to get $\bar{\phi}_n$ and obtain a sequence $\psi_n=\phi_n + \bar\phi_n$ of functions in $C^\infty(G/P)$ such that
 \begin{enumerate}
 \item[(3)] $\lVert\psi_n\rVert_{\hill^{r/2}(K/M)}$ are uniformly bounded for $n\geq1$,
 \item[(4)] $\lim_{n\to \infty} \psi_n(o)=+\infty$,
  \item[(5)] $\mu_0(\psi_n)=0$.
 \end{enumerate}
 Clearly, the existence of such a sequence of functions implies \eqref{proper}. 
 
 Now the problem is local and we can work in a local model $V$ of $G/P$ around the point $o$ using the Cayley transform. To find a desired sequence of functions satisfying $(1)$ and $(2)$, it suffices to find a sequence $\phi_n$ of function in $C^\infty_c(\Omega_0)$ where $\Omega_0$ is a small open neighborhood of the origin $0$ in $V$, satisfying 
  \begin{enumerate}
 \item[(6)] $\lVert\phi_n\rVert_{\dot{\hill}^{r/2}(V)}$ are uniformly bounded for $n\geq1$,
 \item[(7)] $\lim_{n\to \infty} \phi_n(0)=+\infty$.
 \end{enumerate}
To find such a sequence $\phi_n$ on $V$, recall from Lemma \ref{lem_notcont2} that the evaluation map
\[
\mathrm{ev}_0\colon C^\infty_c(V) \to \bC
\]
 at the origin is not continuous with respect to the Sobolev norm for $\dot{\hill}^{r/2}(V)$. The multiplication by a compactly supported, smooth function on $V$ is bounded on the non-homogeneous Sobolev space $\dot{\hill}^{\alpha}(V)$ (see \cite[Theorem 4.15]{Folland1975}). It follows that there is a sequence $\phi_n$ of functions in $C^\infty_c(\Omega_0)$ satisfying $(6)$ and $(7)$ where $\Omega_0$ is an arbitrary small open neighborhood of the origin $0$ and we are done.
 \end{proof}
 
   \end{proof}


 \section{The Busemann cocycle is proper}\label{sec-Busemann}

We consider the Busemann cocycle (see \cite[Section 3.1]{CCJJV2001})
 \[
\gamma_{x,y}(z) = \beta_z(x, y)= \lim_{z'\to z}\bigg{(} d_{Z}(z', y) - d_{Z}(z', x) \bigg{)} 
 \]
 for $x, y$ in $Z=G/K$ and for $z$ in $\partial Z=G/P$. We have the following explicit formula:
 \begin{equation}\label{eq_formula}
\gamma_{x,y}(z) = \log \bigg{|} \frac{q(y, z)}{q(x, z)} \frac{q(x, x)^{1/2}}{q(y, y)^{1/2}} \bigg{|}.
 \end{equation}
The Busemann cocycle $\gamma_{x,y}$ is a smooth function on $G/P$ and the map
\[
(x, y) \to \gamma_{x,y} \in C^\infty(G/P)
\]
defines a $G$-equivariant cocycle $\gamma \colon Z\times Z\to C^\infty(G/P)$. By passing to the quotient space $C^\infty(G/P)/\mathbb{C}1_{G/P}$, we obtain a $G$-equivariant cocycle in $ C^\infty(G/P)/\mathbb{C}1_{G/P}$, which we still denote as $\gamma$. 

\begin{theorem}\label{thm-Busemann}  Let $G$ be any one of $\SO_0(n, 1)$ $n\geq2$, $\SU(n,1)$ for $n\geq2$, and $\Sp(n,1)$ for $n\geq2$. Let $\pi_0$ be the natural representation of $G$ on $C^\infty(G/P)/\mathbb{C}1_{G/P}$. Then, there is a Euclidean norm $\lVert\,\,\rVert$ on $C^\infty(G/P)/\mathbb{C}1_{G/P}$ for which $\pi_0$ is continuous and uniformly bounded and for which the Busemann cocycle $\gamma$ is continuous and proper in sense that
\[
\norm{\gamma(x,y)} \to +\infty \,\,\, \text{as} \,\,\, d_Z(x,y)\to +\infty.
\]
\end{theorem}

\begin{proof}

We equip $C^\infty(G/P)/\bC1_{G/P}$ the Euclidean norm $\norm{\,}$ as in Corollary \ref{cor-so}, \ref{cor-sp}. We know that the representation $\pi_0$ is continuous and uniformly bounded with respect to the norm $\norm{\,}$. It is not hard to see the continuity of the cocycle $\gamma$.

It remains to show the following:
 
\begin{lemma}\label{lem-Busemann} The cocycle $\gamma(x,y)=\mu_y - \mu_x$ in $W_0$ is proper with respect to $\lVert\,\,\rVert$ in a sense that
\[
\lVert \gamma(x, y)\rVert\to +\infty \,\,\, \text{as} \,\,\, d_Z(x,y)\to +\infty.
\]
\end{lemma}

\begin{proof} By double transitivity, for pairs of points with same distance, of the $G$-action on $Z=G/K$ and by the uniform-boundedness of the representation $\pi_0$, it suffices to show that 
\[
\lim_{t\to \infty}\lVert\gamma_{0, a_t0}\rVert_{W_0} = \lim_{t\to \infty}\lVert(1+\Delta_E)^{\frac{r}{4}-\frac12}d_E\gamma_{0, a_t0}\rVert_{\Gamma_{L^2}(E^\ast)} \to +\infty.
\]
By the formula \eqref{eq_formula}, we have
\[
\gamma_{0, a_t0}(z)=  \log \bigg{|} \frac{1-z_n\tanh t}{(1-\tanh^2 t)^{1/2}} \bigg{|} \equiv  \log \bigg{|} {1-z_n\tanh t} \bigg{|}  
\]
in $C^\infty(G/P)/\mathbb{C}1_{G/P}$ for $z=[1, z_1, \cdots, z_n]^T$ in $G/P$. On a small neighborhood $U_o$ of the point $o=[1, 0, \cdots, 0, 1]^T$ of $G/P$, via the Cayley transform $\mathcal{C}$ from $V$ to $G/P$, the norm 
\[
\lVert(1+\Delta_E)^{\frac{r}{4}-\frac12}d_Ew\rVert_{\Gamma_{L^2}(E^\ast)}
\]
for $w$ in $C^\infty_c(U_0)$ is equivalent to the norm
\[
\lVert(1+\Delta_{\mathfrak{o}})^{\frac{r}{4}}w\circ \mathcal{C}\rVert_{L^2(V, d\mu_V)}. 
\]
To show the properness, it is enough to show the following claim: the $L^2$-norm of 
\[
\Delta_{\mathfrak{o}}^{\frac{r}{4}} \bigg{(} \bigg{(} \log \bigg{|} {1-z_n\tanh t} \bigg{|} \bigg{)} \circ \mathcal{C} \bigg{)}
\]
on any small neighborhood of $0$ in $V$ goes to infinity as $t$ goes to infinity. We have
\begin{align*}
& \bigg{(} \log \bigg{|} {1-z_n\tanh t} \bigg{|} \bigg{)} \circ \mathcal{C} \\
= &  \log \bigg{|} 1-\tanh t(1-x^*x/2+y/2)(1+x^*x/2-y/2)^{-1}     \bigg{|} \\
=& \log \bigg{|} (1+x^*x/2-y/2)  -\tanh t\bigg{(} {1-x^*x/2+y/2}\bigg{)}     \bigg{|} - \log \bigg{|} (1+x^*x/2-y/2)\bigg{|}  
\end{align*}
for $(x, y)$ in $V$. The second term does not depend on $t$ and is locally square-integrable at $0$, so we can safely ignore. As $t$ goes to infinity, the first term converges smoothly, except at the point $0$, to the function
\[
\log \bigg{|} (1+x^*x/2-y/2) - \bigg{(} {1-x^*x/2+y/2}\bigg{)}     \bigg{|} = \log \bigg{|} x^*x-y    \bigg{|} 
\]
The function $\Delta_{\mathfrak{o}}^{\frac{r}{4}}  \log \bigg{|} x^*x-y    \bigg{|} $ is homogeneous of degree $-r/2$ and as in Lemma \ref{lem_locint}, such a function is not locally square-integrable at the origin $0$. By Fatou's Lemma, our claim follows so we are done.
\end{proof}
\end{proof}

Alternatively, we may consider the derivative $d_E\gamma_{x,y}$ of the Busemann cocycle in $\Gamma(E^*)$. The previous lemma is equivalent to that this cocycle is proper with respect to the norm $\lVert\,\,\rVert$ given in Theorem \ref{thm_so}, \ref{thm_sp} for which the natural representation $\pi_{E^\ast}$ of $G$ is continuous and uniformly bounded. On the other hand, we note that $\Gamma(E^*)$ has a natural $L^p$-norm for $p=r$, for which the representation $\pi_{E^\ast}$ is isometric (see Proposition \ref{prop_lp} and the proof of Theorem \ref{thm_so}, \ref{thm_sp}).  The (derivative of) Busemann cocycle $\gamma$ is a proper cocycle for a $L^p$-representation as well:

\begin{theorem}\label{thm-Busemann}  Let $G$ be any one of $\SO_0(n, 1)$ $n\geq2$, $\SU(n,1)$ for $n\geq2$, and $\Sp(n,1)$ for $n\geq2$. Let $\pi_{E^\ast}$ be the natural isometric representation of $G$ on $\Gamma_{L^p}(E^\ast)$ for $p=r$. Then, the derivative $d_E\gamma$ of the Busemann cocycle in $\pi_{E^\ast}$ is continuous and proper. In particular, any closed subgroup of simple real-rank-one Lie groups $\SO_0(n, 1)$, $\SU(n,1)$ and $\Sp(n,1)$ admits a metrically proper, isometric affine action on a $L^p$-space $\Gamma_{L^p}(E^\ast)$ for $p=r$.
\end{theorem}
\begin{proof}
It suffices to show the properness:
\[
\lVert d_E\gamma_{x,y}\rVert_{\Gamma_{L^r}(E^\ast)} \to +\infty \,\,\, \text{as} \,\,\, d_Z(x,y)\to +\infty.
\]
This can be proven in the same way as in the proof of Lemma \ref{lem-Busemann}. Using the Cayley transform, it boils down to showing that $d_{\mathfrak{o}}\log \bigg{|} x^*x-y    \bigg{|}$ on $V$ is not locally $L^r$-integrable at the origin $0$ where 
\[
d_{\mathfrak{o}}=\,\,\, \mid_{\mathfrak{o}} \circ d \colon C^\infty(V)\to \Omega^1(V) \to \Gamma(E^*_{\mathfrak{o}})
\]
is the composition of the de-Rham differential $d$ and the restriction of one-forms defined on $TV$ to the subbundle $E_{\mathfrak{o}}$ corresponding to the left-invariant vector field on $V$ defined by $\mathfrak{o}\subset \underline V$. The assertion follows since $d_{\mathfrak{o}}\log \bigg{|} x^*x-y    \bigg{|}$ is homogeneous of degree $-1$ so it is not locally $L^r$-integrable at the origin $0$ as in Lemma \ref{lem_locint}.

\end{proof}


\bibliographystyle{alpha}
\bibliography{Refs}

\end{document}